\documentclass[11pt]{article}
\usepackage[numbers,sort&compress]{natbib}
\usepackage{enumerate}
\usepackage{amscd}
\usepackage{amsmath}
\usepackage{latexsym}
\usepackage{amsfonts}
\usepackage{amssymb}
\usepackage{amsthm}
\usepackage{verbatim}
\usepackage{mathrsfs}
\usepackage{enumerate}
\usepackage{hyperref}

\theoremstyle{plain}\newtheorem{definition}{Definition}[section]
\theoremstyle{definition}\newtheorem{theorem}{Theorem}[section]
\theoremstyle{plain}\newtheorem{lemma}[theorem]{Lemma}
\theoremstyle{plain}\newtheorem{coro}[theorem]{Corollary}
\theoremstyle{plain}
\theoremstyle{remark}\newtheorem{remark}{Remark}[section]
\usepackage{xcolor}

\newcommand{\norm}[1]{\left\|#1\right\|}
\newcommand{\Div}{\mathrm{div}\,}
\newcommand{\B}{\Big}

\newcommand{\be}{\begin{equation}}
\newcommand{\ee}{\end{equation}}
 \newcommand{\ba}{\begin{aligned}}
 \newcommand{\ea}{\end{aligned}}

\newcommand{\fbxoo}{\int_{B_{\varepsilon}(0)}\!\!\!\!\!\!\!\!\!\!\!\!\!\!\!\! -~\,~\,~\,~\,}
\newcommand{\fbxo}{\int_{B_{\varepsilon}(x)}\!\!\!\!\!\!\!\!\!\!\!\!\!\!\!\! -~\,\,~\,~\,}

  \newcommand{\f}{\frac}
    
  \newcommand{\ben}{\begin{enumerate}}
   \newcommand{\een}{\end{enumerate}}

\newcommand{\Rmnum}[1]{\expandafter\@slowromancap\romannumeral #1@}

\allowdisplaybreaks

\numberwithin{equation}{section}
\begin{document}
%%%%%%%%%%%%%%%%%%%%%%%%%%%%%%%%%%%%%%%%%%%%%%%%%%%%%%%%%%%%%%%%%%%%%%%%%%%%%%%%%%%%%%%%%%%%%%%%%%%%
\title{On two conserved quantities  in the inviscid  electron and Hall
magnetohydrodynamic equations}
\author{Yanqing Wang\footnote{   College of Mathematics and   Information Science, Zhengzhou University of Light Industry, Zhengzhou, Henan  450002,  P. R. China Email: wangyanqing20056@gmail.com}   ~
 \, Jing Yang\footnote{College of Mathematics and   Information Science, Zhengzhou University of Light Industry, Zhengzhou, Henan  450002,  P. R. China     Email: cqyj513@outlook.com }~~ and~  Yulin Ye\footnote{Corresponding author. School of Mathematics and Statistics,
Henan University,
Kaifeng, 475004,
P. R. China. Email: ylye@vip.henu.edu.cn}~   }
\date{}
\maketitle
\begin{abstract}
In this paper, we are concerned with the energy and magnetic helicity conservation of weak solutions for both  the  electron  and Hall
magnetohydrodynamic equations. Various   sufficient criteria  to ensure the energy and magnetic helicity conservation  in Onsager's critical spaces $\underline{B}^{\alpha}_{p,VMO}$ and $B^{\alpha}_{p,c(\mathbb{N})}$  in these  systems are established. Moreover, for the E-MHD equations, we  observe that  the conservation criteria of energy and magnetic helicity  to  the E-MHD equations    correspond   to
the helicity and energy to the ideal incompressible Euler equations, respectively.
 \end{abstract}
\section{Introduction}
\label{intro}
\setcounter{section}{1}\setcounter{equation}{0}

The magnetohydrodynamic (MHD) equations,
the electron (E-MHD) and Hall (H-MHD)
magnetohydrodynamic equations play an central role in the theory of plasma physical (see e.g. \cite{[KCY],[Galtier],[Chkhetiani]}).
The inviscid   H-MHD  equations may be written as follows
\be\left\{\ba\label{hallMHD}
&u_{t}+u\cdot\nabla u-h\cdot\nabla h+\nabla\Pi =0, \\
&h_{t}+u\cdot\nabla h-h\cdot\nabla u+\text{d}_{\text{I}}\nabla\times\B[(\nabla\times h)\times h\B]=0, \\
&\Div u=\Div h=0.
 \ea\right.\ee
 Here $u$ represents the velocity field,   $h$ stands for the magnetic field and $\Pi=\pi +\f12 |h|^2$ denotes the magnetic pressure with $\pi$ being the fluid pressure, respectively.
We hereafter denote $\vec{j}=\text{curl}h$ as the electric current.
  Without loss of generality, we  set the iron inertial length $\text{d}_{\text{I}}=1$.  In the present paper, we consider the case of a bounded domain with periodic boundary condition in $\mathbb{R}^d$, namely $\Omega=\mathbb{T}^d$ with $d= 2,3$.
 Formally,  when the velocity $u$ in the  H-MHD  equations \eqref{hallMHD} is zero, this system  becomes the following    E-MHD equations
\be\left\{\ba\label{E-MHD}
&h_{t}+ \nabla\times\B[(\nabla\times h)\times h\B]=0, \\
 & \Div h=0.
 \ea\right.\ee
For the derivation of the E-MHD system and its background, we refer the readers to \cite{[KCY]}.
Moreover, if ignoring the Hall effect, namely letting $\text{d}_{I}=0$, then the Hall-MHD equations \eqref{hallMHD}   reduce to the  standard incompressible MHD equations
\be\left\{\ba\label{MHD}
&u_{t}+u\cdot\nabla u-h\cdot\nabla h+\nabla\Pi =0, \\
&h_{t}+u\cdot\nabla h-h\cdot\nabla u=0, \\
&\Div u=\Div h=0.
 \ea\right.\ee
 As known to all, in physics,  the common conserved quantities in these three systems are the total energy  and magnetic helicity (see \cite{[Galtier],[Chkhetiani]}).
Precisely,
  the   energy   in the E-MHD system \eqref{E-MHD}  and H-MHD (MHD) equations \eqref{hallMHD}  for smooth solutions are
$\f12\int_{\mathbb{T}^d} |h|^{2}dx$ and $\f12\int_{\mathbb{T}^d}|v|^{2}+|h|^{2}dx.$
 The   magnetic helicity  in these system is $\int_{\mathbb{T}^d}  H\cdot {\rm curl\,} H\  dx,
$ where $ H={\rm curl}^{-1}h$ represents the magnetic vector potential.

The mathematical study of  conserved quantities is very close to the Onsager's conjecture on the energy conservation of weak solutions for the ideal Euler equations in \cite{[Onsager]}. The initial Onsager's conjecture is that the critical regularity for weak solutions conserving the energy in the Euler equations   is 1/3 in H\"older spaces.
As the initial Onsager's conjecture for the Euler equations, two parts of the analogous Onsager's conjecture for the usual MHD system \eqref{MHD} are well studied.  For the positive part, the sufficient conditions of the weak solutions implying the energy  and magnetic helicity conservation for the standard MHD system \eqref{MHD}  can be found in \cite{[Yu],[KL],[FLS],[CKS],[WHYL]}. Precisely, a  weak solution $(u,h)$ of the ideal MHD equations \eqref{MHD}  satisfies the total energy conservation if one of the following conditions holds
\begin{itemize}
 \item Caflisch-Klapper-Steele \cite{[CKS]}:
$u\in C(0,T;B_{3,\infty}^{\alpha_{1}}),h\in C(0,T;B_{3,\infty}^{\alpha_{2}}),  \alpha_{1}>1/3,  \alpha_{1}+2\alpha_{2}>1;$
 \item  Kang and Lee \cite{[KL]}: $u\in L^{3}(0,T;B_{3,c(\mathbb{N})}^{\alpha_{1}}),h\in L^{3}(0,T;B_{3,c(\mathbb{N})}^{\alpha_{2}}), \alpha_{1}\geq1/3, \alpha_{1}+2\alpha_{2}\geq1;$
 \item Yu \cite{[Yu]}: $u\in L^{3}(0,T;B_{3,c(\mathbb{N})}^{\alpha_{1}}),h\in L^{3}(0,T;B_{3,\infty}^{\alpha_{2}}), \alpha_{1}\geq1/3, \alpha_{1}+2\alpha_{2}\geq1;$
  \item Wang-Huang-Ye-Liu  \cite{[WHYL]}:  $u\in L^{3} (0,T;\underline{B}^{\alpha}_{3,VMO}
  )\cap L^{p_{1}} (0,T;\underline{B}^{\alpha}_{q_{1},VMO}), h\in L^{p_{2}} (0,T; \dot{B}^{\beta}_{q_{2},\infty})$,
  $\alpha \geq1/3, \alpha +2\beta\geq1,1/q_{1} +2/q_{2}=1,1/p_{1} +2/p_{2}=1,1/p_{2}+1/q_{1}\leq1.$
   \end{itemize}
The magnetic helicity  of  weak solutions for the ideal MHD equations \eqref{MHD} is invariant if
\begin{itemize}
 \item
Caflisch-Klapper-Steele \cite{[CKS]}: $u\in C(0,T;B_{3,\infty}^{\alpha_{1}}),h\in C(0,T;B_{3,\infty}^{\alpha_{2}}), \alpha_{1}+2\alpha_{2}>0;$
 \item Kang and Lee \cite{[KL]}:
$u,h\in L^{3}([0,T];L^{3})$;
 \item Faraco-Lindberg-Sz\'ekelyhidi \cite{[FLS]}: $h \in L^{p_{1}}((0,T) \times\mathbb{T}^d ), h\times   u
\in L^{p_{2}}((0,T) \times\mathbb{T}^d ),$  $1/p_{1}+1/p_{2}=1$;
 \item   Wang-Huang-Ye-Liu \cite{[WHYL]}:
 $u\in L^{p_{1}}(0,T;L^{q_{1}}(\mathbb{T}^d))$ and $h\in L^{p_{2}}(0,T;L^{q_{2}}(\mathbb{T}^d))$, $1\leq p_{1},p_{2},q_{1},q_{2}<\infty$ and $ 1/p_{1}+ 2/p_{2}=1$,$ 1/q_{1}+ 2/q_{2}=1$.
   \end{itemize}
   It is worth pointing out that, for $\alpha >1/3$,  the inclusion relations of the aforementioned Besov spaces are that
\be\label{includ}
C^\alpha \subseteq B^\alpha_{3,\infty}\subseteq B^{\frac{1}{3}}_{3,c(\mathbb{N})}\subseteq \underline{B}^{\frac{1}{3}}_{3,VMO}\subseteq B^{\frac{1}{3}}_{3,\infty},
\ee
where $B^{\frac{1}{3}}_{3,c(\mathbb{N})} $ and $\underline{B}^{\frac{1}{3}}_{3,VMO}
$ are Onsager's critical spaces which were initially introduced by    Cheskidov-Constantin-Friedlander-Shvydkoy in \cite{[CCFS]} and Fjordholm-Wiedemann in \cite{[FW]}, respectively.
Regarding the negative part of Onsager's conjecture, the readers can be referred to \cite{[LZZ],[LZZ1],[MY],[BBV],[CKS],[FLS],[DF]} for the non-uniqueness theory of weak solutions for  both the incompressible ideal and viscous MHD equations.

Now, we turn our attentions back to the ideal E-MHD and H-MHD equations.
Compared with nonlinear term  in terms of convection type of both  the ideal Euler and the standard  MHD equations,  the Hall term $\nabla\times\B[(\nabla\times h)\times h\B]$ in    the E-MHD and H-MHD equations
involves the second order derivative rather than    the first order derivative.
To the knowledge of authors,
there are only two works for positive part of
Onsager type results in the inviscid   E-MHD and H-MHD equations.
Precisely, in  \cite{[KDB]},
Kang-Deng-Bie showed that if
\be\label{KDB}
u,h\in L^{p}\cap B^{\alpha}_{p,\infty}((0,T)\times \mathbb{T}^{3}), ~\vec{j}=\text{curl}h\in B^{\beta}_{p,\infty}((0,T)\times \mathbb{T}^{3}),\f2p+\f1q=1,2\alpha+\beta>1,\ee
 then the total energy  of weak solutions for the  ideal H-MHD equations \eqref{hallMHD} is conserved.
On the other hand,  Dai-Krol-Liu  in \cite{[DKL]} showed that the weak solutions for the ideal E-MHD equations \eqref{E-MHD} preserve
 the energy   if  $h\in L^{3}(0,T;B_{3,c(\mathbb{N})}^{\f23})$ and
magnetic helicity   if
\be\label{dkl0}
h\in L^{3}(0,T;B_{3,c(\mathbb{N})}^{\f13}),\ee
respectively.
The intention of this paper is to establish new energy and
magnetic helicity  conservation criterion       of weak solutions for both the inviscid  E-MHD and
 H-MHD equations.
\begin{theorem}\label{the1.1} Let $d=2,3$ and $h\in C([0,T];L^{2}(\mathbb{T}^d))$ be a weak solution of the incompressible  E-MHD equations  \eqref{E-MHD}.  Then for any $t\in [0,T]$, the  energy  $\f12\int_{\mathbb{T}^d}|h|^{2}dx $ is invariant  provided that one of the following four conditions is satisfied
 \begin{enumerate}[(1)]
 \item $h\in L^{3}(0,T;\underline{B}^{\f23}_{3,VMO}(\mathbb{T}^d))$;
 \item $\vec{j}\in L^{r_{1}}(0,T;\underline{B}^{\alpha}_{p,VMO}(\mathbb{T}^d)),h\in L^{r_{2}}(0,T;B^{\beta}_{q,\infty}(\mathbb{T}^d)),$$\alpha+2\beta\geq1,\f1{r_{1}}+
     \f{2}{r_{2}}=1,\f1p+
     \f2q=1$;
\item $\vec{j}\in L^{r_{1}}(0,T;B^{\alpha}_{p,\infty}(\mathbb{T}^d)),h\in L^{r_{2}}(0,T;\underline{B}^{\beta}_{q,VMO}(\mathbb{T}^d)),$$\alpha+2\beta\geq1,\f1{r_{1}}+
\f{2}{r_{2}}=1,\f1p+
     \f2q=1$;
      \item   $\vec{j}\in L^{3}(0,T;B^{\f13}_{\f{3d}{d+2},c(\mathbb{N})}(\mathbb{T}^d)).$

   \end{enumerate}
\end{theorem}

\begin{remark}
According to inclusion relations \eqref{includ} for Besov spaces, we know that this theorem improves the corresponding results in  \cite{[DKL]}.
\end{remark}
\begin{remark} Since there holds
 $B^{\gamma}_{\ell,c(\mathbb{N})}\subseteq \underline{B}^{\gamma}_{\ell,VMO}$, for any $\gamma \in (0,1) $ and $\ell \in [1,\infty]$, the results in (2) and (3)  are still valid by replacing $\underline{B}^{\gamma}_{\ell,VMO}$  by $B^{\gamma}_{\ell,c(\mathbb{N})}$.
\end{remark}
\begin{remark}
It seems that the energy conservation criteria for the E-MHD equations     is by the analogy of
the helicity conservation roles of the ideal  Euler equations due to Cheskidov-Constantin-Friedlander-Shvydkoy \cite{[CCFS]}.
And the results (2)-(4) are partial motived for the investigation of the helicity  of the ideal Euler equations by Chae in \cite{[Chae1],[Chae]} and by De Rosa \cite{[De Rosa]}.
\end{remark}
Before we state the result concerning the magnetic helicity for E-MHD equations \eqref{E-MHD}, we recall that the magnetic vector potential equations  take the form below
\be\label{1.8}
H_{t}+(\nabla\times h)\times h+\nabla p=0, ~\text{div} H=0,
\ee
where $H={\rm curl}^{-1}h$.
\begin{theorem}\label{the1.2} Let $d=3$ and $ \vec{j} \in C(0,T;L^{\f32}(\mathbb{T}^{3}))$
  be a  weak solution of  the inviscid  E-MHD equations  \eqref{E-MHD}. If $h\in L^{3}(0,T;\underline{B}^{\f13}_{3,VMO}(\mathbb{T}^3)),$
 then for any $t\in [0,T]$, the magnetic   helicity   is invariant, i.e.
 $$\int_{\mathbb{T}^3}  H\cdot {\rm curl\,} H\  dx=\int_{\mathbb{T}^3}  H_0\cdot {\rm curl\,} H_0\  dx,$$
 where $H={\rm curl}^{-1}h$.
\end{theorem}
\begin{remark}
Owing to
  inclusion relations   \eqref{includ}, we see that this theorem generalizes the sufficient condition \eqref{dkl0} for the conservation of  magnetic   helicity.
\end{remark}
\begin{remark}
It seems that the magnetic helicity conservation criteria for the E-MHD equations corresponds  to
the energy conservation roles of the ideal Euler equations Cheskidov-Constantin-Friedlander-Shvydkoy \cite{[CCFS]}.
\end{remark}
Next, we formulate our results on inviscid Hall-MHD equations \eqref{hallMHD} as follows.
\begin{theorem}\label{energythe1.1}
Let $d=2,3$ and the pair $(u,h)\in C([0,T];L^{2}(\mathbb{T}^d))$ be a weak solution of the   inviscid Hall-MHD equations  \eqref{hallMHD}. Then the total energy  $\f12\int_{\mathbb{T}^d}(|u|^{2}+|h|^{2})dx $  is invariant  provided that $$u\in L^{3}(0,T;\underline{B}^{\f13}_{3,VMO}(\mathbb{T}^d))\ \text{and}\ h\in L^{3}(0,T;\underline{B}^{\f23}_{3,VMO}(\mathbb{T}^d)). $$
\end{theorem}
Similar as \eqref{1.8}, an analogue of  the magnetic vector potential equations     for \eqref{hallMHD} can be  given by
\be\label{1.10}
H_{t}+h\times u+(\nabla\times h)\times h+\nabla p=0, ~\text{div} H=0,
\ee
where $H={\rm curl}^{-1} h$. Then the sufficient conditions for the conservation of magnetic helicity  involving \eqref{hallMHD}  and \eqref{1.10} are stated as follows.
\begin{theorem}\label{mhthe1.3} Let    $d=3$ and $ h\in C([0,T];L^3(\mathbb{T}^3))$ be a  weak solution of  the inviscid Hall-MHD equations \eqref{hallMHD}. If
  $u\in L^{3} (0,T;L^{3}(\mathbb{T}^{3}))$ and $ h\in L^{3}(0,T;\underline{B}^{\f13}_{3,VMO}(\mathbb{T}^3)),$
 then for any $t\in [0,T]$, the magnetic   helicity  is constant, i.e.
  $$\int_{\mathbb{T}^3}  H\cdot {\rm curl\,} H\  dx=\int_{\mathbb{T}^3}  H_0\cdot {\rm curl\,} H_0\  dx,$$
 where $H={\rm curl}^{-1}h$.
\end{theorem}

\begin{remark}
It seems that this is the first result for weak solutions keeping magnetic   helicity   for  the inviscid  Hall-MHD equations \eqref{hallMHD}.
\end{remark}
Next, we are concerned with the energy and  the generalized magnetic   helicity in the   viscous E-MHD equations and the viscous H-MHD equations. Indeed, for the following viscous   E-MHD equations
\be\left\{\ba\label{visE-MHD}
&h_{t}+ \nabla\times\B[(\nabla\times h)\times h\B]-\Delta h=0, \\
& \Div h=0,
\ea\right.\ee
in \cite{[DKL]},  Dai-Krol-Liu  proved that for any $\phi(x,t)\in \mathcal{D}([0,T]\times \mathbb{R}^3)$ and $t\in (0,T]$,
the following generalized magnetic   helicity equation holds \be\label{mhi}\ba
&\int_{\mathbb{R}^{3}\times \{t\}} h\cdot H \phi dx +2\int^{t}_{0}\int_{\mathbb{R}^{3}} \nabla H :\nabla h \phi dxds\\
=&\int_{\mathbb{R}^{3}\times \{0\}} h\cdot H \phi dx+\int^{t}_{0}\int_{\mathbb{R}^{3}} h\cdot(\phi_t +\Delta \phi)dxds+\int^{t}_{0}\int_{\mathbb{R}^{3}}((\nabla \times h)\times h)\cdot (\nabla \phi \times H) dxds,\ea\ee
if
\be\label{dkl3}
h\in L^{3}(0,T;L^{\f92}(\mathbb{R}^{3})) ~~\text{and}~~  \vec{j}\in L^{3}(0,T;L^{\f95}(\mathbb{R}^{3}))\cap L^{3}(0,T;L^{\f{18}{5}}(\mathbb{R}^{3})) \ \text{ outside a}\  C^{\f12}\ \text{curve}.
\ee
The main difference between the inviscid and viscous equations is that the second one
obeys the additional regularity  $L^{2}(0,T;H^{1}(\mathbb{T}^{3}))$.
Partially inspired by the recent work \cite{[WY]},
as a consequence  result of Theorem \ref{the1.1} and   Theorem \ref{the1.2},
one can get the    energy and   magnetic helicity equality criterion  for the viscous E-MHD equations \eqref{visE-MHD} as follows.
\begin{coro}\label{coro1.6}Let $h   \in L^{\infty}(0,T;L^{2}(\mathbb{T}^d))\cap L^{2}(0,T;H^{1}(\mathbb{T}^d))$ be a  weak solutions of \eqref{visE-MHD}. \begin{enumerate}[1.]
		\item
		Then the   energy  $\f12\int_{\mathbb{T}^d} |h|^{2}dx $  is invariant  provided that $$   \vec{j}\in L^{r_1}(0,T;L^p(\mathbb{T}^d))\ \text{and}\  h\in L^{r_2}(0,T;L^q(\mathbb{T}^d)), \f{2}{r_1}+\f{1}{r_2}=1, \f{2}{p}+\f{1}{q}=1. $$
		\item Let $d=3$, then the following  magnetic   helicity equality holds that
		$$\int_{\mathbb{T}^3}h\cdot H(x,T) dx-\int_{\mathbb{T}^3} h\cdot H(x,0) dx +2\int_0^T\int_{\mathbb{T}^3}(\nabla \times h)\cdot h dxdt=0,$$
		provided that one of the following conditions is satisfied
	\begin{enumerate}[(1)]
		\item  $\vec{j}\in L^{p}(0,T;L^{q}(\mathbb{T}^{3})),h\in L^{\f{2p}{p-1}}(0,T;L^{\f{2q}{q-1}}(\mathbb{T}^{3})),$
	%	\item $\vec{j}\in L^{3}(0,T;L^{\f95}(\mathbb{T}^{3})),$
		\item$h\in L^{4}(0,T;L^{4}(\mathbb{T}^{3}));$
		\item$h\in L^{p}(0,T;L^{q}(\mathbb{T}^{3}))	\ \text{with}\ \f{1}{p}+
		\f{3}{q}=1, \ \text{if}\ 3<q< 4;$ or $\f{2}{p}+
		\f{2}{q}=1,  \ \text{if}\  q\geq 4;$
		\item   $\vec{j}\in L^{p}(0,T;L^{q}(\mathbb{T}^{3}))
		\ \text{with}\ \frac{1}{p}+\frac{3 }{q}=2,  \ \text{if}\ \frac{3 }{2}<q<\frac{9}{5};$
	 or $\frac{1}{p}+\frac{6}{5 q}=1,    \ \text{if}\  q\geq \frac{9}{5}.$
	\end{enumerate}
\end{enumerate}
\end{coro}
\begin{remark}
	Notice that \eqref{dkl3} is a special case of this theorem.
\end{remark}
\begin{remark}
	This corollary extends the famous Lion-Shinbrot type energy balance   criteria of the 3D Navier-Stokes equations to the magnetic helicity conservation of the 3D viscous E-MHD equations  \eqref{visE-MHD}.
\end{remark}
\begin{remark}
Inspired by \cite{[WMH]}, it should be pointed out that this corollary still holds for the whole spaces $\mathbb{R}^{3}$.
\end{remark}
Moreover,  Dumas and   Sueur in  \cite{[DS]} showed that   the weak solutions of
  the viscous H-MHD equations below
 \be\left\{\ba\label{vhallMHD}
&u_{t}+u\cdot\nabla u-h\cdot\nabla h+\nabla\Pi -\Delta u=0, \\
&h_{t}+u\cdot\nabla h-h\cdot\nabla u+ \nabla\times\B[(\nabla\times h)\times h\B]-\Delta h=0, \\
&\Div u=\Div h=0,
 \ea\right.\ee
  conserve
 the energy   if  \be
 u\in L^{3}(0,T;B_{3,c(\mathbb{N})}^{\f13}) ~~\text{and}~~ h\in L^{3}(0,T;B_{3,c(\mathbb{N})}^{\f23}),\ee
  and preserve the
magnetic helicity   if
\be\label{dsl0}
h\in L^{3}(0,T;B_{3,c(\mathbb{N})}^{\f13}).\ee
In the spirt of the above theorems and the classical  work \cite{[CCFS]}, one can refine   Dumas and   Sueur's result as follows.
\begin{coro}\label{coro1.5}
Let   $(u,h)   \in L^{\infty}(0,T;L^{2}(\mathbb{T}^3))\cap L^{2}(0,T;H^{1}(\mathbb{T}^3)))$ be a weak solution of the    viscous  Hall-MHD equations  \eqref{vhallMHD}. \begin{enumerate}[(1)]
 \item
 Then the total energy  $\f12\int_{\mathbb{T}^3}(|u|^{2}+|h|^{2})dx $  is invariant  provided that $$u\in L^{3}(0,T;B^{\f13}_{3,\infty}(\mathbb{T}^3))\ \text{and}\ h\in L^{3}(0,T;\underline{B}^{\f23}_{3,VMO}(\mathbb{T}^3)) $$
 \item
The following  magnetic   helicity equality holds that
$$\int_{\mathbb{T}^3}h\cdot H(x,T) dx-\int_{\mathbb{T}^3} h\cdot H(x,0) dx +2\int_0^T\int_{\mathbb{T}^3}(\nabla \times h)\cdot h dxdt=0,$$
provided
$h\in L^{3}(0,T;B^{\f13}_{3,\infty}(\mathbb{T}^3))$.
\end{enumerate}
\end{coro}

Finally, we would like to mention that
four-thirds law of conserved quantities  in the inviscid  electron and Hall
magnetohydrodynamic systems  can be found in \cite{[WC]}.
Moreover, Dai  and   Liu recently constructed the
weak solutions to the nonresistive
E-MHD system \eqref{E-MHD} which do not conserve energy and magnetic
helicity in \cite{[DL]}.    For the non-unique weak solutions in Leray-Hopf class for the three-dimensional Hall-MHD system, the reader may refer to   \cite{[D]}.

The paper is organized as follows. In section 2, we present some notations used in this paper and the key lemma concerning the Constantin-E-Titi type commutator  estimates for the functions in  Besov-VMO space $\underline{B}^{\alpha}_{p,VMO}$ and   homogeneous Besov space $\dot{B}^{\beta}_{p,\infty}$. Section 3 and  section 4 are devoted to the proof of the energy and magnetic helicity conservation for the inviscid  E-MHD equations \eqref{E-MHD} and H-MHD equations \eqref{hallMHD}, respectively.
 Finally, in Section 5, we consider the conserved quantities  in the   viscous E-MHD equations \eqref{visE-MHD} and  Hall-MHD equations \eqref{vhallMHD}.

\section{Notations and some auxiliary lemmas} \label{section2}

First, we introduce some notations used in this paper.
 For $p\in [1,\,\infty]$, the notation $L^{p}(0,\,T;X)$ stands for the set of measurable functions $f$ on the interval $(0,\,T)$ with values in $X$ and $\|f\|_{X}$ belonging to $L^{p}(0,\,T)$.

Second,
for $s\in\mathbb{R}$ and $1\leq p \leq \infty$, we define the Besov semi-norm $\norm{f}_{\dot{B}^s_{p,\infty}(\mathbb{T}^{d})}$ and Besov norm $\norm{f}_{B^s_{p,\infty}(\mathbb{T}^{d})}$ of $f\in \mathcal{S}^{'}$ as
$$\ba
&\|f\|_{\dot{B}_{p, \infty}^s(\mathbb{T}^d)}=\||y|^{-s}\| f(x-y)-f(x)\|_{L_x^p(\mathbb{T}^d)}\|_{L_y^{\infty}(\mathbb{T}^d)},\\
&\norm{f}_{B^s_{p,\infty}(\mathbb{T}^{d})}=\norm{f}_{{L^p}(\mathbb{T}^{d})}
+\norm{f}_{\dot{B}^s_{p,\infty}(\mathbb{T}^{d})}.
\ea$$
In the spirit of  \cite{[FW]},  we  recall
the Besov-VMO space  $L^p(0,T;\underline{B}^{\alpha}_{q,VMO}(\mathbb{T}^d))$  of function $f$ if it satisfies
$$\|f\|_{L^p(0,T;L^q(\mathbb{T}^d))}<\infty,$$
and
$$\ba
&\lim_{\varepsilon\rightarrow0}\f{1}{\varepsilon^{\alpha}}\left(\int_0^T\B[\int_{\mathbb{T}^d} \fbxo|f(x)-f(y)|^{q}dydx \B]^{\f{p}{q}}dt\right)^{\f1p}\\
=&\lim_{\varepsilon\rightarrow0}\f{1}{\varepsilon^{\alpha}}\left(\int_0^T\B[\int_{\mathbb{T}^d} \fbxoo|f(x)-f(x-y)|^{q}dydx \B]^{\f{p}{q}}dt\right)^{\f{1}{p}}=0.
\ea$$

Eventually, we let $\eta_{\varepsilon}:\mathbb{R}^{d}\rightarrow \mathbb{R}$ be a standard mollifier, i.e. $\eta(x)=C_0e^{-\frac{1}{1-|x|^2}}$ for $|x|<1$ and $\eta(x)=0$ for $|x|\geq 1$, where $C_0$ is a constant such that $\int_{\mathbb{R}^d}\eta (x) dx=1$. For $\varepsilon>0$, we define the rescaled mollifier by  $\eta_{\varepsilon}(x)=\frac{1}{\varepsilon^d}\eta(\frac{x}{\varepsilon})$, and for  any function $f\in L^1_{\rm loc}(\mathbb{R}^d)$, its mollified version is defined by
\be\label{mollified}
f^\varepsilon(x)=(f*\eta_{\varepsilon})(x)=\int_{\mathbb{R}^d}f(x-y)\eta_{\varepsilon}(y)dy,\ \ x\in \mathbb{R}^d.
\ee

Next, we state several useful lemmas which will be frequently used  in the proof of present paper.
\begin{lemma}\label{lem2.1}(\cite{[WWY],[BGSTW]})
Let $\alpha \in (0,1)$,  $ p,q\in [1,\infty]$,  and $k\in \mathbb{N}^+$. Assume that  $f\in L^p(0,T;\dot{B}^\alpha_{q,\infty}(\mathbb{T}^d))$ and $g\in L^p(0,T; \underline{B}^{\alpha}_{q,VMO} (\mathbb{T}^d))$,   then  letting $\varepsilon\to 0$, there hold that
 \begin{enumerate}[(1)]
 \item $ \|f^{\varepsilon} -f \|_{L^{p}(0,T;L^{q}(\mathbb{T}^d))}\leq \text{O}(\varepsilon^{\alpha})\|f\|_{L^p(0,T;\dot{B}^\alpha_{q,\infty}(\mathbb{T}^d))}$;
   \item   $ \|\nabla^{k}f^{\varepsilon}  \|_{L^{p}(0,T;L^{q}(\mathbb{T}^d))}\leq \text{O}(\varepsilon^{\alpha-k})
       \|f\|_{L^p(0,T;\dot{B}^\alpha_{q,\infty}(\mathbb{T}^d))}$;
 \item $ \|g^{\varepsilon} -g \|_{L^{p}(0,T;L^{q}(\mathbb{T}^d))}\leq \text{o}(\varepsilon^{\alpha}) $;
   \item   $ \|\nabla^{k}g^{\varepsilon}  \|_{L^{p}(0,T;L^{q}(\mathbb{T}^d))}\leq \text{o}(\varepsilon^{\alpha-k}) .$
\end{enumerate}
\end{lemma}

In the following, we will state a generalized Constantin-E-Titi type commutator estimates involving the Besov-VMO spaces. The special case $p=q=3$ and $\alpha=\beta=\f13$ in  (1) of the following lemma was presented in \cite{[BGSTW]} and the following version is due to
\cite{[WWY]}. For the convenience of readers, we outline the proof to make the paper more self-contained.
\begin{lemma}(\cite{[WWY]})	\label{lem2.3}
	Assume that $0<\alpha,\beta<1$, $1\leq p,q,p_{1},p_{2}\leq\infty$ and $\frac{1}{p}=\frac{1}{p_1}+\frac{1}{p_2}$.
Then, there holds
	\begin{align} \label{cet}
		\|(fg)^{\varepsilon}- f^{\varepsilon}g^{\varepsilon}\|_{L^p(0,T;L^q(\mathbb{T}^d))} \leq  \text{o}(\varepsilon^{\alpha+\beta}),\ \text{as} \ \varepsilon\to 0,	
	\end{align}
provided that one of the following three conditions is satisfied,
\begin{enumerate}[(1)]
\item  $f\in L^{p_1}(0,T;\underline{B}^{\alpha}_{q_1,VMO} (\mathbb{T}^d) )$, $g\in L^{p_2}(0,T;\underline{B}^{\beta}_{q_2,VMO}(\mathbb{T}^d)  )$, $1\leq q_{1},q_{2}\leq\infty$, $\frac{1}{q}=\frac{1}{q_1}+\frac{1}{q_2}$;

 \item  $f\in L^{p_1}(0,T;\underline{B}^{\alpha}_{q_{1},VMO}(\mathbb{T}^d))$, $g\in L^{p_2}(0,T;\dot{B}^{\beta}_{q_{2},\infty} (\mathbb{T}^d))$, $1\leq q_{1},q_{2}\leq\infty$, $\frac{1}{q}=\frac{1}{q_1}+\frac{1}{q_2},q_{2}\geq \f{q_{1}}{q_{1}-1}$ and $p_{2}\geq\f{q_{1}}{q_{1}-1}$.
 \end{enumerate}\end{lemma}

\begin{proof}
The key point of this lemma is that the following
 Constantin-E-Titi identity   in \cite{[CET]}  
\be\label{CETI}\ba&(fg)^{\varepsilon}(x)- f^{\varepsilon}g^{\varepsilon}(x)\\
=&
\int_{B_{\varepsilon}(0)}\eta_{\varepsilon}(y)
\B[f(x-y)-f(x)\B]\B[g(x-y)-g(x)\B]dy-
(f-f^{\varepsilon})(g-g^{\varepsilon})(x)\\
=&\,I+II.
	\ea\ee
	(1) It follows from H\"older's inequality  that
$$
\ba
&|I|\\
 \leq&
 C\B[\fbxoo
|f(x-y)-f(x)|^{q_{1}} dy\B]^{\f{1}{q_{1}}}\B[\fbxoo
|g(x-y)-g(x)|^{q_{2}} dy\B]^{\f{1}{q_{2}}}.\ea$$
Performing  a space-time integration and using the definition of Besov-VMO space, as $\varepsilon\to0 $,  we note that
\be\label{key0}\ba
&\|I\|_{L^p(0,T;L^q(\mathbb{T}^d))}\\
\leq & C\B\|\left( \int_{\mathbb{T}^d}\fbxoo
|f(x-y)-f(x)|^{q_{1}} dy\right)^{\f{1}{q_{1}}}\B\|_{L^{p_1}(0,T)} \B\| \left(\int_{\mathbb{T}^d}\fbxoo
|g(x-y)-g(x)|^{q_{1}} dy\right)^{\f{1}{q_{2}}}\B\|_{L^{p_2}(0,T)}\\
\leq& o(\varepsilon^{\alpha+\beta}),
\ea\ee
where we need to require $\f{1}{p}=\f{1}{p_1}+\f{1}{p_2}, \f{1}{q}=\f{1}{q_1}+\f{1}{q_2}$.

We deduce from the  H\"older inequality and Lemma \ref{lem2.1} that, as $\varepsilon\to 0$,  
\be\label{key2}\begin{aligned}  &\|II\|_{L^p(0,T;L^q(\mathbb{T}^d))} 
	\leq& C\| f-f^{\varepsilon}\|_{L^{p_1}(0,T;L^{q_1}(\mathbb{T}^d))}  \| g-g^{\varepsilon}\|_{L^{p_2}(0,T;L^{q_2}(\mathbb{T}^d))} 
	\leq& o(\varepsilon^{\alpha+\beta}).
	\end{aligned} \ee
Then substituting \eqref{key2} and \eqref{key0} into \eqref{CETI}, we finish the proof of this part.

(2) The  H\"older inequality implies that
$$|I|\leq \f{1}{\varepsilon^{d}}\B[\int_{B_{\varepsilon}(0)}
\big|\eta \big(\f{y}{\varepsilon}\big)
 [g(x-y)-g(x)]\big|^{\f{q_{1}}{q_{1}-1}}dy
 \B]^{1-\f{1}{q_{1}}}\B[\int_{B_{\varepsilon}(0)}
|f(x-y)-f(x)|^{q_{1}} dy\B]^{\f{1}{q_{1}}}.
$$
Combining this,  the  H\"older inequality again and Minkowski inequality, we find
$$\ba
&\|I\|_{L^{q}(\mathbb{T}^d)}\\ \leq&
C\B[\f{1}{\varepsilon^d}\int_{B_{\varepsilon}(0)}
\big
 \| g(x-y)-g(x) \|_{L^{q_{2}}(\mathbb{T}^d)} ^{\f{q_{1}}{q_{1}-1}}dy\B]^{1-\f{1}{q_{1}}}\B[\int_{\mathbb{T}^d}\fbxoo
|f(x-y)-f(x)|^{q_{1}} dydx\B]^{\f{1}{q_{1}}},
\ea$$
where we used $\f{1}{q}=\f{1}{q_1}+\f{1}{q_2}$ and $q_{2}\geq \f{ q_{1}}{q_{1}-1}$.
Thereby, by utilizing H\"older inequality and Minkowski inequality once again, as $\varepsilon\to0$, we observe that
\be\label{key1}\ba & \|I\|_{L^p(0,T;L^q(\mathbb{T}^d))}\\ \leq&C \B[\f{1}{\varepsilon^d}\int_{B_{\varepsilon}(0)}
 \| g(x-y)-g(x) \|_{L^{p_{2}}(0,T;L^{q_{2}}(\mathbb{T}^d))} ^{\f{q_{1}}{q_{1}-1}}dy\B]^{1-\f{1}{q_{1}}}
 \B\| \int_{\mathbb{T}^d}\fbxoo
|f(x-y)-f(x)|^{q_{1}} dydx\B\|_{L^{p_1}(0,T)}\\
  \leq& o(\varepsilon^{\alpha+\beta}).
\ea\ee
Plugging  \eqref{key1} and \eqref{key2} into \eqref{CETI}, we get the desired estimate. The proof of this lemma is completed. 
\end{proof}
\begin{lemma}(\cite{[WWY]})	\label{lem2.5}
	Assume that $0<\alpha,\beta<1$, $1\leq p,q,p_{1},p_{2}\leq\infty$ and $\frac{1}{p}=\frac{1}{p_1}+\frac{1}{p_2}$. Let  $f^{\varepsilon}$ and $g^{\varepsilon}$ be defined by \eqref{mollified}
	Then as $\varepsilon\to 0$,  there holds
	\begin{align} \label{cet2}
		\|(fg)^{\varepsilon}- f^{\varepsilon}g^{\varepsilon}\|_{L^p(0,T;L^q(\mathbb{T}^d))} \leq  \text{o}(\varepsilon^{\alpha+\beta}),	
	\end{align}
	provided that one of the following three conditions is satisfied,
	\begin{enumerate}[(1)]
		\item  $f\in L^{p_1}(0,T;\dot{B}^{\alpha}_{q_{1},c(\mathbb{N})} (\mathbb{T}^d))$, $g\in L^{p_2}(0,T;\dot{B}^{\beta}_{q_{2},\infty} (\mathbb{T}^d))$, $1\leq q_{1},q_{2}\leq\infty$, $\frac{1}{q}=\frac{1}{q_1}+\frac{1}{q_2}$;
		\item  $\nabla f\in   L^{p_1}(0,T;\dot{B}^{\alpha}_{q_{1},c(\mathbb{N})} (\mathbb{T}^d))$, $\nabla g\in L^{p_2}(0,T;\dot{B}^{\beta}_{q_{2},\infty}(\mathbb{T}^d) )$,  $\f{2}{d}+\f1q=\frac{1}{q_{1}}+\frac{1}{q_{2}}$, $1\leq q_{1},q_{2}<d$;
		\item  $  f\in   L^{p_1}(0,T;\dot{B}^{\alpha}_{q_{1},c(\mathbb{N})} (\mathbb{T}^d))$, $\nabla g\in L^{p_2}(0,T;\dot{B}^{\beta}_{q_{2},\infty} (\mathbb{T}^d))$,  $\f{1}{d}+\f1q=\frac{1}{q_{1}}+\frac{1}{q_{2}}$, $1\leq q_{2}<d$,  $1\leq q_{1}\leq\infty$.
\end{enumerate}\end{lemma}
 \begin{lemma}\label{lem2.4}(\cite{[YWW],[WWY]})
	Let $ p,q,p_1,q_1,p_2,q_2\in[1,+\infty)$ with
$\frac{1}{p}=\frac{1}{p_1}+\frac{1}{p_2},\frac{1}{q}=\frac{1}{q_1}+\frac{1}{q_2} $. Assume $f\in L^{p_1}(0,T;L^{q_1}(\mathbb{T}^d)) $ and $g\in
L^{p_2}(0,T;L^{q_2}(\mathbb{T}^d))$, then as $\varepsilon\to0$, it holds
	\begin{equation}\label{a4}
	\|(fg)^\varepsilon-f^\varepsilon
g^\varepsilon\|_{L^p(0,T;L^q(\mathbb{T}^d))}\rightarrow 0,
	\end{equation}
and
\begin{equation}\label{b7}
	\|(f\times g)^\varepsilon-f^\varepsilon\times g^\varepsilon\|_{L^p(0,T;L^q(\mathbb{T}^d))}\rightarrow 0.
\end{equation}
\end{lemma}

 \begin{lemma}(\cite{[NNT]})\label{lem2.6}
Suppose that $f\in L^{p}(0,T;L^{q}(\mathbb{T}^{d}))$. Then for any $\varepsilon>0$, there holds
\be\label{}
\|\nabla f^{\varepsilon}\|_{L^{p}(0,T;L^{q}(\mathbb{T}^{d}))}
\leq C\varepsilon^{-1}\| f\|_{L^{p}(0,T;L^{q}(\mathbb{T}^{d}))},
\ee
and, if $p,q<\infty$
$$
\limsup_{\varepsilon\rightarrow0} \varepsilon\|\nabla f^{\varepsilon}\|_{L^{p}(0,T;L^{q}(\mathbb{T}^{d}))}=0.
$$

\end{lemma}
 \begin{lemma}(\cite{[NNT],[WY]}) \label{lem2.7}   Let $1\leq p,q,p_1,p_2,q_1,q_2\leq \infty$  with $\frac{1}{p}=\frac{1}{p_1}+\frac{1}{p_2}$ and $\frac{1}{q}=\frac{1}{q_1}+\frac{1}{q_2}$. Assume $f\in L^{p_1}(0,T;W^{1,q_1}(\mathbb{T}^d))$ and $g\in L^{p_2}(0,T;L^{q_2}(\mathbb{T}^d))$. Then for any $\varepsilon> 0$, there holds
		\begin{align} \label{fg'}
		\|(fg)^\varepsilon-f^\varepsilon g^\varepsilon\|_{L^p(0,T;L^q(\mathbb{T}^d))}\leq C\varepsilon \|f\|_{L^{p_1}(0,T;W^{1,q_1}( \mathbb{T}^d))}\|g\|_{L^{p_2}(0,T;L^{q_2}(\mathbb{T}^d))}.
		\end{align}
		Moreover, if $p_2,q_2<\infty$ then
		\begin{align}\label{limite'}
		\limsup_{\varepsilon \to 0}\varepsilon^{-1} \|(fg)^\varepsilon-f^\varepsilon g^\varepsilon\|_{L^p(0,T;L^q(\mathbb{T}^d))}=0.
		\end{align}
	\end{lemma}
Finally, for the convenience of readers, we state the definitions of weak solutions for both the inviscid E-MHD equations \eqref{E-MHD} and H-MHD equations \eqref{hallMHD}, respectively. To do this, we list some identities first.  For any vectors $\vec{A}$ and $\vec{B}$, there hold that
\be\ba\label{VI}
&\nabla(\vec{A}\cdot \vec{B})=\vec{A}\cdot\nabla \vec{B}+\vec{B}\cdot\nabla \vec{A}+\vec{A}\times(\nabla \times\vec{B})+\vec{B}\times(\nabla\times\vec{A}),\\
& \nabla\times(\vec{A}\times \vec{B})=\vec{A} \text{div} \vec{B}-\vec{B} \text{div}  \vec{A}+\vec{B}\cdot\nabla \vec{A}-\vec{A}\cdot\nabla \vec{B},\\
& \Div (\nabla \times \vec{A})=0,\ \nabla \times (\nabla \vec{B})=0,
\ea\ee
which  together with the divergence-free condition $\Div u= \Div h=0$
allow us to get
\be\label{identity}\ba
&h\cdot\nabla h = \f{1}{2}\nabla |h|^{2}+\vec{j}\times h,\ \text{with} \ \vec{j}=\nabla \times h,\\
&\nabla\times( h\times u)=u\cdot\nabla h-h\cdot\nabla u,\ea\ee
which means that
\begin{align}
	\label{non2}&\nabla\times\B[(\nabla\times h)\times h\B]= \nabla\times[\vec{j}\times h]=\nabla\times\B[\text{div}(h\otimes h)-\nabla\f12|h|^{2}\B]=\nabla\times\B[\text{div}(h\otimes h)\B].
\end{align}
Hence, we get a equivalent form  of the inviscid E-MHD equation \eqref{E-MHD} as
\begin{align}
	\label{eq2}
	& h_{t}+\nabla\times\B[\text{div}(h\otimes h)\B]=0.
\end{align}
Moreover, we rewrite magnetic   potential equation \eqref{1.8}  as
\be\label{dengjiahp1}
H_{t}+\text{div}(h\otimes h)+\nabla p=0.
\ee
Thanks to $\eqref{VI}_2$ and $\eqref{non2}$, we reformulate $\eqref{hallMHD}_{2}$ in the Hall-MHD equations as
\be\label{dengjiaE-MHD1}
h_{t}+\nabla\times( h\times u)+\nabla\times\B[\text{div}(h\otimes h)\B]=0,
\ee
which implies its  magnetic potential equations
\be\label{12}
H_{t}+h\times u+\text{div}(h\otimes h)+\nabla p=0.
\ee
Based on this, we present the definitions of weak solutions for the incompressible inviscid H-MHD equations  and  its magnetic potential equations \eqref{12}. The definitions of weak solutions for incompressible inviscid E-MHD equations  and  its magnetic potential equations  can be given in a similar way, hence we omit the details.

\begin{definition}[Weak solutions of incompressible inviscid H-MHD equations]   \label{weaksolu}
	The pair  $(u, b) \in C_{{\rm weak}}(0,T; L^2(\mathbb{T}^d))$ with initial data $u_0, b_0\in L^2(\mathbb{T}^d)$ is a weak solution to the incompressible inviscid H-MHD equations \eqref{hallMHD}  if	
	\begin{enumerate}[(1)]
		\item For all $t\in [0,T]$, $(u(t,x), h(t,x))$ is divergence-free in the sense of distributions, namely,
		$$\int_{0}^{T}\int_{\mathbb{T}^d}u\cdot\nabla \phi dx=0,\quad\int_{0}^{T}\int_{\mathbb{T}^d}h\cdot\nabla \phi dx=0,$$
		for any smooth function $\phi(x,t)\in C_0^{\infty} ([0,T]\times\mathbb{T}^d)$
		\item Equations  hold in the sense of distributions, i.e.,
		for any divergence-free test vector field $\varphi  \in C_0^{\infty} ((0,T) \times \mathbb{T}^d)$,
		\begin{align*}
			&\int_0^T \int_{\mathbb{T}^d}  \partial_t \varphi \cdot u    + \nabla\varphi :( u \otimes u - h \otimes h) d x d t =0,
			\\
			&\int_0^T	\int_{\mathbb{T}^d}  \partial_t \varphi \cdot h   + \nabla\varphi  :(h \otimes u- u \otimes h)+(h \otimes h):\nabla\nabla\times\varphi d x dt  = 0 .
		\end{align*}
	\end{enumerate}
\end{definition}
Then we can define the weak solution of magnetic potential  equations \eqref{12} as follows.
\begin{definition}[Weak solutions of the magnetic potential  equations for inviscid H-MHD equations]\label{eulerdefi2}
	$H\in C_{{\rm weak}}(0,T;L^2(\mathbb{T}^d))$ with initial data in $L^2(\mathbb{T}^d)$ is called  a weak solution of  equation \eqref{12}, if
	for any divergence-free test vector field $\varphi\in C_{0}^{\infty}((0,T)\times\mathbb{T}^d)$, there holds
	$$
	\int_{0}^{T}\int_{\mathbb{T}^d}H\cdot\partial_{t}\varphi+(u\times h)\varphi+(h\otimes h):\nabla\varphi dxdt=0.
	$$		
\end{definition}

\section{ Two conserved quantities in the inviscid E-MHD equations}
In this section, we commence the proof of the energy and maganetic helicity conservation for the inviscid E-MHD equations \eqref{E-MHD}.
 First, we show the energy conservation of weak solutions for the E-MHD equations.
\begin{proof}[Proof of Theorem \ref{the1.1}]
	In what follows, for the sake of simplicity, we assume that   the weak solution $h$ of the inviscid E-MHD equations \eqref{E-MHD} is differential in time and only space mollification is applied.  This ensures that  the weak solution $h$ satisfies
	\be\label{p2}
	h_{t}^{\varepsilon} +\nabla\times\B[(\nabla\times h)\times h\B]^{\varepsilon}=0,
	\ee
	which holds pointwise in space and time derivative is a classical derivative at Lebesgue almost all times. For the rigorous argument, we refer the reader to \cite[Section 2. p.741-p.744]{[DE]} by Drivas and Eyink.
 Multiplying $\eqref{p2}$ by $h^\varepsilon$, then integrating it   over $(0,T)\times \mathbb{T}^3$, we have
\be\label{key3.2}\int_0^T\int_{\mathbb{T}^{d}} \f{1}{2}\f{d}{dt}|h^{\varepsilon}|^2dxdt+\int_0^T\int_{\mathbb{T}^d}\nabla\times\B[(\nabla\times h)\times h\B]^{\varepsilon}h^{\varepsilon}dxdt=0.\ee
To obtain the energy balance of weak solutions, we need to take the limits on both sides of equation \eqref{key3.2} as $\varepsilon \to 0$. For the convenience of proof, we denote the second term on the left-hand side of \eqref{key3.2} by $L$.
First, in light of  \eqref{non2} and the integration by parts, we know that
\begin{equation}\label{c2}
	\begin{aligned}
		&\int_0^T\int_{\mathbb{T}^d} \nabla\times\B[(\nabla\times h)\times h\B]^{\varepsilon}h^{\varepsilon}dxdt\\
		=&	\int_0^T\int_{\mathbb{T}^{d}}\nabla\times\B[\text{div}(h\otimes h)\B]^{\varepsilon}h^{\varepsilon}dxdt	\\
		=&-\int_0^T\int_{\mathbb{T}^{d}}\B[  (h\otimes h)^{\varepsilon}-(h^{\varepsilon}\otimes h^{\varepsilon}) \B]\nabla(\nabla\times h^{\varepsilon})dxdt-\int_0^T\int_{\mathbb{T}^d} (h^\varepsilon\otimes h^{\varepsilon})\nabla(\nabla\times h^{\varepsilon})dxdt\\
		=&-\int_0^T\int_{\mathbb{T}^{d}}\B[  (h\otimes h)^{\varepsilon}-(h^{\varepsilon}\otimes h^{\varepsilon}) \B]\nabla(\nabla\times h^{\varepsilon})dxdt,
	\end{aligned}
\end{equation}
where we have used the divergence-free condition and the following facts
\begin{equation}\label{c3}
	\begin{aligned}
		\int_0^T\int_{\mathbb{T}^{d}}(h^{\varepsilon} \otimes   h^{\varepsilon} )\nabla(\nabla\times h^{\varepsilon})dxdt=&-\int_0^T\int_{\mathbb{T}^{d}}\B[\text{div}(h^{\varepsilon}\otimes h^{\varepsilon})-\nabla \f12|h^{\varepsilon}|^{2} \B]\nabla\times h^{\varepsilon}dxdt\\
		=&-\int_0^T\int_{\mathbb{T}^{d}}\B[(\nabla\times h^{\varepsilon})\times h^{\varepsilon}\B]\nabla\times h^{\varepsilon}dxdt\\
		=&-\int_0^T\int_{\mathbb{T}^{d}}\B[\vec{j}^\varepsilon\times h^{\varepsilon}\B]\cdot \vec{j}^{\varepsilon}dxdt=0,		
	\end{aligned}
\end{equation}
and \begin{equation}
	\vec{A}\times \vec{B}\cdot \vec{A}=\vec{A}\times \vec{A}\cdot \vec{B}=0, \ \text{for\ any\ vectors}\  \vec{A}\ \text{and}\ \vec{B}.
\end{equation}
Then, combining \eqref{c2} and \eqref{key3.2}, we have
\be\ba\label{3.8}
&\int_0^T\int_{\mathbb{T}^{d}}\f{1}{2}\f{d}{dt}|h^\varepsilon|^2dxdt-\int_0^T\int_{\mathbb{T}^{d}}\B[  (h\otimes h)^{\varepsilon}-(h^{\varepsilon}\otimes h^{\varepsilon}) \B]\nabla(\nabla\times h^{\varepsilon})dxdt=0.
\ea\ee
 Now, we are in a position to show that the   term $L$ tends to zero as $\varepsilon\to 0$. Indeed, thanks to the H\"older inequality, we see that
\begin{equation}\label{c4}
|L|\leq C   \| (h^{\varepsilon}\otimes h^{\varepsilon})- (h\otimes h)^{\varepsilon} \|_{L^{\f{3}{2}}(0,T;L^{\f{3}{2}}(\mathbb{T}^d))}\|\nabla(\nabla\times h^{\varepsilon})\|_{L^{3}(0,T;L^{3}(\mathbb{T}^d))}
\end{equation}
Since $h\in L^{3}(0,T;\underline{B}^{\f23}_{3,VMO}(\mathbb{T}^d))$, then applying Lemma  \ref{lem2.3}, we observe that, as $\varepsilon\rightarrow0$,
\begin{equation}\label{c5}
\| (h^{\varepsilon}\otimes h^{\varepsilon})- (h\otimes h)^{\varepsilon} \|_{L^{\f{3}{2}}(0,T;L^{\f{3}{2}}(\mathbb{T}^d))}\leq o(\varepsilon^{\f{4}{3} }).
\end{equation}
Furthermore, it follows from Lemma \ref{lem2.1} that, as $\varepsilon\rightarrow0$,
\begin{equation}\label{c6}\|\nabla(\nabla\times h^{\varepsilon})\|_{L^{3}(0,T;L^{3}(\mathbb{T}^d))}\leq o(\varepsilon^{-\f{4}{3}}).\end{equation}
As a consequence, combining \eqref{c4}, \eqref{c5} and \eqref{c6}, we have
\be\label{3.3}
|L|=\B|\int_0^T\int_{\mathbb{T}^{d}}\nabla\times\B[(\nabla\times h)\times h\B]^{\varepsilon}h^{\varepsilon}dxdt
\B|\leq o(1).
\ee
Thus, passing to the limit of $\varepsilon$ in \eqref{key3.2}, we conclude the desired energy conservation.

(2)
Recalling \eqref{key3.2}, the second term $L$ can be rewriten as
\begin{equation}\label{c7}\ba
\int_0^T\int_{\mathbb{T}^{d}}\nabla\times\B[(\nabla\times h)\times h\B]^{\varepsilon}h^{\varepsilon}dxdt
=&\int_0^T\int_{\mathbb{T}^{d}}\nabla\times(\vec{j}\times h)^{\varepsilon}h^{\varepsilon}dxdt\\
=&\int_0^T\int_{\mathbb{T}^{d}}(\vec{j}\times h)^{\varepsilon}\nabla\times h^{\varepsilon}dxdt.
\ea\end{equation}
Notice that
\begin{equation}\label{c8}\int_0^T\int_{\mathbb{T}^{d}}(\vec{j}^{\varepsilon}\times h^{\varepsilon})\nabla\times h^{\varepsilon}dxdt=\int_0^T\int_{\mathbb{T}^{d}}(\vec{j}^{\varepsilon}\times h^{\varepsilon})\cdot\vec{j}^{\varepsilon}dxdt=0.
\end{equation}
Hence, combining \eqref{c7} and \eqref{c8}, we get
\begin{equation}\label{c9}\ba
L=\int_0^T\int_{\mathbb{T}^{d}}\nabla\times\B[(\nabla\times h)\times h\B]^{\varepsilon}h^{\varepsilon}dxdt
=\int_0^T\int_{\mathbb{T}^{d}}\B[(\vec{j}\times h)^{\varepsilon}-(\vec{j}^{\varepsilon}\times h^{\varepsilon})\B]\nabla\times h^{\varepsilon}dxdt.
\ea\end{equation}
Now, we need to show the term $L$ tends to zero as $\varepsilon \to 0$. First, by virtue of the H\"older inequality, we discover that
\begin{equation}\label{c10}\ba
&\B|\int_0^T\int_{\mathbb{T}^d}\nabla\times\B[(\nabla\times h)\times h\B]^{\varepsilon}h^{\varepsilon}dxdt
\B|\\
\leq& \|(\vec{j}\times h)^{\varepsilon}-(\vec{j}^{\varepsilon}\times h^{\varepsilon})\|_{L^{\f{r_1 r_2}{r_1+r_2}}(0,T;L^{\f{pq}{p+q}}(\mathbb{T}^d))}
\|\nabla\times h^{\varepsilon}\|_{L^{r_{2}}(0,T;L^{q}(\mathbb{T}^d))},
\ea\end{equation}
where we need to require that $\f{1}{r_1}+\f{2}{r_2}=1$ and $\f{1}{p}+\f{2}{q}=1$.
On the other hand, by means of Lemma \ref{lem2.3}, we deduce from $\vec{j}\in L^{r_{1}}(0,T;\underline{B}^{\alpha}_{p,VMO}(\mathbb{T}^d)),h\in L^{r_{2}}(0,T;B^{\beta}_{q,\infty}(\mathbb{T}^d))$ that
\begin{equation}\label{c11}\|(\vec{j}\times h)^{\varepsilon}-(\vec{j}^{\varepsilon}\times h^{\varepsilon})\|_{L^{\f{r_1 r_2}{r_1+r_2}}(0,T;L^{\f{pq}{p+q}}(\mathbb{T}^d))}\leq o(\varepsilon^{\alpha+\beta}), \ \text{as}\ \varepsilon\to 0.
\end{equation}
Moreover, letting $\varepsilon\to 0$ and applying  Lemma \ref{lem2.1}, we deduce that
\begin{equation}\label{c12}
\|\nabla\times h^{\varepsilon}\|_{L^{r_{2}}(0,T;L^{q}(\mathbb{T}^d))}\leq O(\varepsilon^{\beta-1}).
\end{equation}
Then substituting  \eqref{c11} and \eqref{c12} into \eqref{c10},  we can obtain
$$|L|=\B|\int_0^T\int_{\mathbb{T}^{d}}\nabla\times\B[(\nabla\times h)\times h\B]^{\varepsilon}h^{\varepsilon}dxdt
\B| \leq o(\varepsilon^{\alpha+2\beta-1}),\ \text{as}\ \varepsilon\to 0.
$$
Therefore, a combination of $\alpha+2\beta-1\geq0$ and \eqref{key3.2} yields the desired results.

(3)  A slight modified the  proof of (2) allows us to complete the proof of this part. We omit the details here.

(4) In the same manner of derivation of \eqref{c2}, due to $\vec{j}=\nabla \times h$, we get
\begin{equation}\label{c13}\ba
\int_0^T\int_{\mathbb{T}^{d}}\nabla\times\B[(\nabla\times h)\times h\B]^{\varepsilon}h^{\varepsilon}dxdt
=\int_0^T\int_{\mathbb{T}^{d}}\B[ (h^{\varepsilon}\otimes h^{\varepsilon})- (h\otimes h)^{\varepsilon} \B]\nabla \vec{j}^{\varepsilon} dxdt.
\ea\end{equation}
Then in the light of the H\"older inequality, we arrive at
\begin{equation}\label{c14}\ba
&\B|\int_0^T\int_{\mathbb{T}^{d}}\nabla\times\B[(\nabla\times h)\times h\B]^{\varepsilon}h^{\varepsilon}dxdt
\B|\\
\leq& \| (h^{\varepsilon}\otimes h^{\varepsilon})- (h\otimes h)^{\varepsilon} \|_{L^{\f32}(0,T;L^{\f{3d}{2d-2}}(\mathbb{T}^{d}))}\|\nabla \vec{j}^{\varepsilon} \|_{L^{3}(0,T;L^{\f{3d}{d+2}} (\mathbb{T}^{d}))}.
\ea\end{equation}
With the help of (3) in Lemma   \ref{lem2.5} and $\vec{j}\in L^{3}(0,T;B^{\f13}_{\f{3d}{d+2},c(\mathbb{N})}(\mathbb{T}^d))$,  we infer that, as $\varepsilon\rightarrow0$,
$$\| (h^{\varepsilon}\otimes h^{\varepsilon})- (h\otimes h)^{\varepsilon} \|_{L^{\f32}(0,T;L^{\f{3d}{2d-2}}(\mathbb{T}^{d}))}\leq o(\varepsilon^{\f{2}{3}}).
$$
Moreover, by means of  Lemma   \ref{lem2.1}, we know that, as $\varepsilon\rightarrow0$,
$$\|\nabla \vec{j}^{\varepsilon} \|_{L^{3}(0,T;L^{\f{3d}{d+2}} (\mathbb{T}^{d}))}\leq o(\varepsilon^{ -\f{2}{3}}),$$
which turns out that
$$
\B|\int_0^T\int_{\mathbb{T}^{d}}\nabla\times\B[(\nabla\times h)\times h\B]^{\varepsilon}h^{\varepsilon}dxdt
\B|\leq o(1).
$$
Consequently, we immediately finish the proof of this part by $\vec{j}\in L^{3}(0,T;B^{\f13}_{\f{3d}{d+2},c(\mathbb{N})}(\mathbb{T}^d))$.

\end{proof}
Next, we turn our attentions to the proof of magnetic helicity balance for the inviscid E-MHD equations \eqref{E-MHD}.
\begin{proof}[Proof of Theorem \ref{the1.2}]
Recalling    the magnetic   potential equation for
E-MHD  equations \eqref{dengjiahp1} that
\be\label{mpeforE-MHD}
H_{t} +\text{div}(h\otimes h)  +\nabla p=0,
\ee
where $H=\text{curl}^{-1}h$.
Mollifying the equations \eqref{mpeforE-MHD} and \eqref{E-MHD}  in spatial direction, we observe that
$$\ba
H_{t}^{\varepsilon}  +\text{div}(h\otimes h)^{\varepsilon}  +\nabla p^{\varepsilon}  =0,\\
h_{t}^{\varepsilon}  +\nabla\times\B[\text{div}(h\otimes h)\B]^{\varepsilon} =0.
\ea$$
Then it follows from a straightforward computation that
\begin{align}\label{3.6}
\f{d}{dt}\int_{\mathbb{T}^3}h^{\varepsilon}\cdot H^{\varepsilon}dx=&
\int_{\mathbb{T}^3}h_{t}^{\varepsilon}\cdot H^{\varepsilon}dx+
\int_{\mathbb{T}^3}h^{\varepsilon}\cdot H_{t}^{\varepsilon}dx\nonumber\\
=&-\int_{\mathbb{T}^3}  \nabla\times\B[\text{div}(h\otimes h)\B]^{\varepsilon} \cdot H^{\varepsilon}dx -
\int_{\mathbb{T}^3}h^{\varepsilon}\cdot \B[ \text{div}(h\otimes h)^{\varepsilon}   +\nabla p^{\varepsilon}\B]dx\nonumber\\
=&-2\int_{\mathbb{T}^3}\B[\text{div}(h\otimes h)\B]^{\varepsilon}h^{\varepsilon}dx\nonumber\\
=&-2\int_{\mathbb{T}^3}\B[\text{div}(h\otimes h)^{\varepsilon}-\Div(h^\varepsilon\otimes h^\varepsilon)\B]h^{\varepsilon}dx\nonumber\\
=&2\int_{\mathbb{T}^3}\B[(h\otimes h)^{\varepsilon}-(h^\varepsilon\otimes h^\varepsilon)\B]\nabla h^{\varepsilon}dx,
\end{align}
where in the fourth equality we have used the divergence-free condition that $\Div h=0$.
Hence, integrating the above equation over $[0,T]$, we know that
\begin{equation}\label{03.6}
 \int_{\mathbb{T}^3} h ^{\varepsilon}(x,T)\cdot H^{\varepsilon}(x,T) dx -\int_{\mathbb{T}^3} h ^{\varepsilon}(x,0)\cdot H^{\varepsilon}(x,0) dx=2\int_0^T\int_{\mathbb{T}^3}\B[(h\otimes h)^{\varepsilon}-(h^\varepsilon\otimes h^\varepsilon)\B]\nabla h^{\varepsilon}dxdt.
\end{equation}
For the convenience of computation, we denote the term on the right-hand side of above equation \eqref{03.6} by $R$. Consequently, to obtain the desired magnetic helicity equality, we need to show that the limit of term $R$ is zero as $\varepsilon\to 0$.
 Indeed, by the H\"older inequality, we have
\begin{equation}\label{c18}\ba
|R|
= &\B|2\int_{0}^{T}\int_{\mathbb{T}^3}\B[ (h\otimes h)^{\varepsilon}-h^{\varepsilon}\otimes h^{\varepsilon}\B]\nabla h^{\varepsilon}dxdt\B|\\
\leq&C\|(h\otimes h)^{\varepsilon}-h^{\varepsilon}\otimes h^{\varepsilon}\|_{L^{\f32}(0,T;L^{\f32}(\mathbb{T}^{3}))}\|\nabla h^{\varepsilon}\|_{L^{3}(0,T;L^{3}(\mathbb{T}^{3}))}.
\ea\end{equation}
Since $h\in L^{3}(0,T;\underline{B}^{\f13}_{3,VMO}(\mathbb{T}^3)),$ then aplying  Lemma \ref{lem2.3} we have, as $\varepsilon
\to0$,
$$\|(h\otimes h)^{\varepsilon}-h^{\varepsilon}\otimes h^{\varepsilon}\|_{L^{\f32}(0,T;L^{\f32}(\mathbb{T}^{3}))}\leq o(\varepsilon^{\f{2}{3}})
$$
On the other hand, by virtue  of Lemma \ref{lem2.1}, as $\varepsilon\to 0$,  we find
$$\|\nabla h^{\varepsilon}\|_{L^{3}(0,T;L^{3}(\mathbb{T}^{3}))}\leq o(\varepsilon^{-\f{2}{3}})$$
Hence, letting $\varepsilon\to 0$, we conclude that
\be\label{0.37}
|R|
= \B|2\int_{0}^{T}\int_{\mathbb{T}^3}\B[ (h\otimes h)^{\varepsilon}-h^{\varepsilon}\otimes h^{\varepsilon}\B]\nabla h^{\varepsilon}dxdt\B|\leq o(1).
\ee
Thus, taking the limits in \eqref{03.6} as  $\varepsilon\rightarrow0$  and using \eqref{0.37}, we finish the proof of this theorem.

\end{proof}

 \section{Energy and magnetic helicity conservation for the inviscid  H-MHD system}
 In this section, we are concerned with the energy and magnetic helicity conservation for the inviscid H-MHD equations.
\begin{proof}[Proof of Theorem \ref{energythe1.1}]
  Mollifying the equations \eqref{hallMHD}  in spatial direction (see the notations in Section 2), we have
\be\ba\label{rmhd}
&\partial_{t}{u^{\varepsilon}} +  \text{div}(u\otimes u)^{\varepsilon} - \text{div}(h\otimes h)^{\varepsilon} +\nabla\Pi^{\varepsilon}= 0,  \\
&\partial_{t}{h^{\varepsilon}} + \text{div}(u\otimes h)^{\varepsilon} - \text{div}(h\otimes u)^{\varepsilon}+\nabla\times\B[(\nabla\times h)\times h\B]^{\varepsilon}  = 0. \\
\ea\ee
Multiplying $\eqref{rmhd}_1$ by $u^\varepsilon$ and $\eqref{rmhd}_2$ by $h^\varepsilon$ respectively, then  integrating them over $[0,T]\times \mathbb{T}^d$, we have
\begin{equation}\ba\label{4.2}
& \f12\|u^{\varepsilon}(T)\|^{2}_{L^{2}(\mathbb{T}^d)}
+\f12\|h^{\varepsilon}(T)\|^{2}_{L^{2}(\mathbb{T}^d)}
-\f12\|u^{\varepsilon}(0)\|^{2}_{L^{2}(\mathbb{T}^d)}
-\f12\|h^{\varepsilon}(0)\|^{2}_{L^{2}(\mathbb{T}^d)}\\=&-\int_{0}^{T}\int_{\mathbb{T}^d}
\text{div}(u\otimes u)^{\varepsilon}u^{\varepsilon}-\text{div}(h\otimes h)^{\varepsilon}u^{\varepsilon}+ \text{div}(\Pi u^{\varepsilon})
+\text{div}(u\otimes h)^{\varepsilon}h^{\varepsilon}\\&-\text{div}(h \otimes u)^{\varepsilon} h^{\varepsilon}+\nabla\times\B[(\nabla\times h)\times h\B]^{\varepsilon}h^{\varepsilon}  dxdt\\=&-\int_{0}^{T}\int_{\mathbb{T}^d}
\text{div}(u\otimes u)^{\varepsilon}u^{\varepsilon}-\text{div}(h\otimes h)^{\varepsilon}u^{\varepsilon}
+\text{div}(u\otimes h)^{\varepsilon}h^{\varepsilon}\\&-\text{div}(h \otimes u)^{\varepsilon} h^{\varepsilon}+\nabla\times\B[\text{div}(h\otimes h)\B]^{\varepsilon}h^{\varepsilon}  dxdt\\
=&\,I+II+III+IV+V,
\ea\end{equation}
where the divergence-free condition and \eqref{non2} have been used.

First, by means of divergence-free condition and the integration by parts, the term $I$ can be rewritten as
\begin{equation}\label{c22}\ba
I
=&-\int_{\mathbb{T}^d}
\text{div}\big((u\otimes u)^{\varepsilon}-
(u^{\varepsilon}\otimes u^{\varepsilon})\big)u^{\varepsilon}dxdt-
\int_{\mathbb{T}^d}\text{div}(u^{\varepsilon}\otimes u^{\varepsilon})u^{\varepsilon}dxdt\\
=&-\int_{\mathbb{T}^d}
\text{div}\big((u\otimes u)^{\varepsilon}-
(u^{\varepsilon}\otimes u^{\varepsilon})\big)u^{\varepsilon}dxdt\\
=&\int_{\mathbb{T}^d}
\big((u\otimes u)^{\varepsilon}-
(u^{\varepsilon}\otimes u^{\varepsilon})\big)\nabla u^{\varepsilon}dxdt.
\ea\end{equation}
 Likewise, for the second term $II$, the integration by parts helps us to get
\begin{equation}\label{c23}\ba
	II=& \int_{0}^{T}\int_{\mathbb{T}^d}
	\text{div} \big((h\otimes h)^{\varepsilon}- (h^{\varepsilon}\otimes h^{\varepsilon})\big)u^{\varepsilon} dxdt+\int_{0}^{T} \int_{\mathbb{T}^d} \text{div} (h^{\varepsilon}\otimes h^{\varepsilon})u^{\varepsilon}dxdt\\
	& =-\int_{0}^{T}\int_{\mathbb{T}^d}
	\big((h\otimes h)^{\varepsilon}- (h^{\varepsilon}\otimes h^{\varepsilon})\big)\nabla u^{\varepsilon} dxdt+\int_{0}^{T} \int_{\mathbb{T}^d} \text{div} (h^{\varepsilon}\otimes h^{\varepsilon})u^{\varepsilon}dxdt\\
	=&\,II_{1}+II_{2}.
	\ea\end{equation}
We will show that the terms $II_{2}$ can be   canceled  in term $IV$. Indeed,
in light of the integration by parts and $\text{div}\,h=0$, we arrive at
\begin{equation}\label{c24}\ba
	IV=& \int_{0}^{T}\int_{\mathbb{T}^d}
	\text{div} \big((h\otimes u)^{\varepsilon}-(h^{\varepsilon}\otimes u^{\varepsilon})\big)h^{\varepsilon}
	dxdt+\int_{0}^{T}\int_{\mathbb{T}^d}\text{div} (h^{\varepsilon}\otimes u^{\varepsilon})h^{\varepsilon}dxdt
	\\
	=&-\int_{\mathbb{T}^d}
	\big((h\otimes u)^{\varepsilon}-(h^{\varepsilon}\otimes u^{\varepsilon})\big)\nabla h^{\varepsilon}
	dxdt-\int_{0}^{T}\int_{\mathbb{T}^d} h^\varepsilon\cdot\nabla h^\varepsilon\cdot  u^\varepsilon dxdt\\
	=&-\int_{0}^{T}\int_{\mathbb{T}^d}
	\big((h\otimes u)^{\varepsilon}-(h^{\varepsilon}\otimes u^{\varepsilon})\big)\nabla h^{\varepsilon}
	dxdt-\int_{0}^{T}\int_{\mathbb{T}^d} \text{div}  ( h^\varepsilon \otimes h^\varepsilon) u^\varepsilon dxdt\\=&\,IV_{1}-II_{2}. \ea\end{equation}
Next, by the same token, we can control the term $III$ as
\begin{equation}\label{c25}
	\begin{aligned}
		III=&-\int_{0}^{T} \int_{\mathbb{T}^d}
		\text{div} \big((u\otimes h)^{\varepsilon}-(u^{\varepsilon } \otimes h^{\varepsilon})\big)h^{\varepsilon}dxdt-
		\int_{0}^{T} \int_{\mathbb{T}^d}\text{div} (u^{\varepsilon } \otimes h^{\varepsilon})h^{\varepsilon}dxdt\\
		=&\int_{0}^{T} \int_{\mathbb{T}^d}
		\big((u\otimes h)^{\varepsilon}-(u^{\varepsilon } \otimes h^{\varepsilon})\big)\nabla h^{\varepsilon}dxdt.
	\end{aligned}
\end{equation}
Finally,  for the  term $V$, recalling \eqref{c2}, we can obtain
\begin{equation}\label{c26}
	\begin{aligned}
		V
		=&	\int_0^T\int_{\mathbb{T}^{d}}\nabla\times\B[\text{div}(h\otimes h)\B]^{\varepsilon}h^{\varepsilon}dxdt	\\
		=&-\int_0^T\int_{\mathbb{T}^{d}}\B[  (h\otimes h)^{\varepsilon}-(h^{\varepsilon}\otimes h^{\varepsilon}) \B]\nabla(\nabla\times h^{\varepsilon})dxdt.
	\end{aligned}
\end{equation}
Then substrituting \eqref{c22}-\eqref{c26} into \eqref{4.2}, we have
\begin{equation}\label{c27}
	\begin{aligned}
	& \f12\|u^{\varepsilon}(T)\|^{2}_{L^{2}(\mathbb{T}^d)}
	+\f12\|h^{\varepsilon}(T)\|^{2}_{L^{2}(\mathbb{T}^d)}
	-\f12\|u^{\varepsilon}(0)\|^{2}_{L^{2}(\mathbb{T}^d)}
	-\f12\|h^{\varepsilon}(0)\|^{2}_{L^{2}(\mathbb{T}^d)}\\	
	=&\int_{\mathbb{T}^d}
	\big((u\otimes u)^{\varepsilon}-
	(u^{\varepsilon}\otimes u^{\varepsilon})\big)\nabla u^{\varepsilon}dxdt-\int_{0}^{T}\int_{\mathbb{T}^d}
	\big((h\otimes h)^{\varepsilon}- (h^{\varepsilon}\otimes h^{\varepsilon})\big)\nabla u^{\varepsilon} dxdt\\
	&+\int_{0}^{T} \int_{\mathbb{T}^d}
	\big((u\otimes h)^{\varepsilon}-(u^{\varepsilon } \otimes h^{\varepsilon})\big)\nabla h^{\varepsilon}dxdt-\int_{0}^{T}\int_{\mathbb{T}^d}
	\big((h\otimes u)^{\varepsilon}-(h^{\varepsilon}\otimes u^{\varepsilon})\big)\nabla h^{\varepsilon}
	dxdt\\&-\int_0^T\int_{\mathbb{T}^{d}}\B[  (h\otimes h)^{\varepsilon}-(h^{\varepsilon}\otimes h^{\varepsilon}) \B]\nabla(\nabla\times h^{\varepsilon})dxdt\\
	=&I+II_1+III+IV_1+V.
	\end{aligned}
	\end{equation}
Now, to obtain the desired energy conservation, we are in a position to prove the terms on the right-hand side of \eqref{c27} tend to zero as $\varepsilon\to 0$. However, since the last term $V$ is same as that in the proof of Theorem \ref{the1.1}, here we oly focus on the rest terms.
Actually,  for the term $I$, by virtue of the H\"older inequality, we can get
\be\ba\label{3.4}
 |I|
 \leq& C \|(u^{\varepsilon}\otimes u^{\varepsilon}) -(u\otimes u)^{\varepsilon}\|_{L^{\f32}(0,T;L^{\f32}(\mathbb{T}^d))}
 \|\nabla u^{\varepsilon}\|_{L^{3}(0,T;L^{3}(\mathbb{T}^d))}.
 \ea\ee
On the other hand, applying  Lemma \ref{lem2.3} and Lemma \ref{lem2.1} with $u\in L^{3}(0,T;B^{\f13}_{3,VMO}(\mathbb{T}^d))$, we discover
 $$\|(u^{\varepsilon}\otimes u^{\varepsilon}) -(u\otimes u)^{\varepsilon}\|_{L^{\f32}(0,T;L^{\f32}(\mathbb{T}^d))}\leq o(\varepsilon^{\f{2}{3}}),$$
and
 $$\|\nabla u^{\varepsilon}\|_{L^{3}(0,T;L^{3}(\mathbb{T}^d))}\leq o(\varepsilon^{-\f{2}{3}}),$$
which together with \eqref{3.4} imply
\be\label{3.5}
|I|\leq o(1),\ \text{as}\ \varepsilon\to 0.
\ee
 To control the terms $II_1$, $III$ and $IV_1$, in light of the H\"older inequality, we have
\begin{equation}\label{c28}\ba
&|II_1+ III+IV_1|\\
\leq  &\int_{0}^{T}\int_{\mathbb{T}^d}
  |\big((h\otimes h)^{\varepsilon}- (h^{\varepsilon}\otimes h^{\varepsilon})\big)||\nabla u^{\varepsilon}| dxdt+ \int_{0}^{T}\int_{\mathbb{T}^d}
|\big((h\otimes u)^{\varepsilon}-(h^{\varepsilon}\otimes u^{\varepsilon})\big)||\nabla h^{\varepsilon}|
dxdt\\
\leq &C\| (u^{\varepsilon }\otimes h^{\varepsilon}) -
(u\otimes h)^{\varepsilon} \|_{L^{\f{3}{2}}(0,T;L^{\f{3}{2}}(\mathbb{T}^d))}
\|\nabla h^{\varepsilon}\|_{L^{3}(0,T;L^{3}(\mathbb{T}^d))}\\
\leq &C\|(h\otimes h)^{\varepsilon}- (h^{\varepsilon}\otimes h^{\varepsilon})\|_{L^{\f{3}{2}}(0,T;L^{\f{3}{2}}(\mathbb{T}^d))}
 \|\nabla u^{\varepsilon}\|_{L^{3}(0,T;L^{3}(\mathbb{T}^d))}\\&+C\|(h\otimes u)^{\varepsilon}-(h^{\varepsilon}\otimes u^{\varepsilon})\|_{L^{\f{3}{2}}(0,T;L^{\f{3}{2}}(\mathbb{T}^d))}
 \|\nabla h^{\varepsilon}\|_{L^{3}(0,T;L^{3}(\mathbb{T}^d))},
\ea\end{equation}
Then due to $h\in L^{3}(0,T;\underline{B}^{\f23}_{3,VMO}(\mathbb{T}^d))$ and $u\in L^{3}(0,T;\underline{B}^{\f13}_{3,VMO}(\mathbb{T}^d))$, by virtue of  (1) in Lemma \ref{lem2.3}, as $\varepsilon\to 0$, we remark that
$$\| (h^{\varepsilon }\otimes h^{\varepsilon}) -
(h\otimes h)^{\varepsilon} \|_{L^{\f{3}{2}}(0,T;L^{\f{3}{2}}(\mathbb{T}^d))}\leq o(\varepsilon^{ \f43 }),$$
$$\| (h^{\varepsilon }\otimes u^{\varepsilon}) -
(h\otimes u)^{\varepsilon} \|_{L^{\f{3}{2}}(0,T;L^{\f{3}{2}}(\mathbb{T}^d))}\leq o(\varepsilon^{ 1 }),$$
$$\| (u^{\varepsilon }\otimes h^{\varepsilon}) -
(u\otimes h)^{\varepsilon} \|_{L^{\f{3}{2}}(0,T;L^{\f{3}{2}}(\mathbb{T}^d))}\leq o(\varepsilon^{ 1 }).$$
And it follows from Lemma
\ref{lem2.1} that
$$\|\nabla h^{\varepsilon}\|_{L^{3}(0,T;L^{3}(\mathbb{T}^d))}\leq o(\varepsilon^{-\f13}),\ \text{and}\ \|\nabla u^{\varepsilon}\|_{L^{3}(0,T;L^{3}(\mathbb{T}^d))}\leq o(\varepsilon^{-\f23}).$$
Thus, we end up with
$$ |II_1+ III+IV_1|\leq o(\varepsilon^{\f23}).$$
Recalling
\eqref{3.3}, we have
$$
|V|\leq o(1). $$
With these estimates in hand, by taking $\varepsilon\rightarrow0$ in \eqref{4.2}, we can achieve the  proof of  Theorem \ref{energythe1.1}.

\end{proof}
Next, we consider the magnetic helicity conservation for the inviscid H-MHD equations \eqref{hallMHD}.
\begin{proof}[Proof of Theorem \ref{mhthe1.3}]
Mollifying the equations  \eqref{dengjiaE-MHD1} and \eqref{12} in spatial direction, we
have
$$\ba
&h_{t}^{\varepsilon} +\nabla\times( h\times u)^{\varepsilon} +\nabla\times\B[\text{div}(h\otimes h)\B]^{\varepsilon} =0,\\
 &H_{t}^{\varepsilon}  + ( h\times u)^{\varepsilon} +\text{div}(h\otimes h)^{\varepsilon}   +\nabla p^{\varepsilon}   =0.
 \ea$$
By a straightforward computation, it gives
\begin{equation}\label{4.5}\ba
&\f{d}{dt}\int_{\mathbb{T}^3}h^{\varepsilon}\cdot H^{\varepsilon}dx=
\int_{\mathbb{T}^3}h_{t}^{\varepsilon}\cdot H^{\varepsilon}dx+
\int_{\mathbb{T}^3}h^{\varepsilon}\cdot H_{t}^{\varepsilon}dx\\
=&-\int_{\mathbb{T}^3} \B[\nabla\times( h\times u)^{\varepsilon} +\nabla\times\B[\text{div}(h\otimes h)\B]^{\varepsilon} \B]\cdot H^{\varepsilon}dx-
\int_{\mathbb{T}^3}h^{\varepsilon}\cdot \B[( h\times u)^{\varepsilon}  +\text{div}(h\otimes h)^{\varepsilon}   +\nabla p^{\varepsilon}\B]dx\\
=&-\int_{\mathbb{T}^3}\B[(h\times u)^\varepsilon+\Div (h\otimes h)^\varepsilon\B]\cdot h^\varepsilon dx-\int_{\mathbb{T}^3}h^\varepsilon\cdot \B[(h\times u)^\varepsilon+\Div (h\otimes h)^\varepsilon\B]dx\\
=&-2\int_{\mathbb{T}^3}(h\times u)^\varepsilon \cdot h^\varepsilon dx-2\int_{\mathbb{T}^3} \Div (h\otimes h)^\varepsilon \cdot h^\varepsilon dx,
\ea\end{equation}
where the integration by parts and the divergence-free condition have been used.
Then we conclude by performing a time integral   that
\begin{equation}\label{c30}\ba
&\int_{\mathbb{T}^3} h ^{\varepsilon}(x,T)\cdot B^{\varepsilon}(x,T) dx -\int_{\mathbb{T}^3} h ^{\varepsilon}(x,0)\cdot B^{\varepsilon}(x,0) dx\\=&-2\int_{0}^{T}\int_{\mathbb{T}^3}(h\times u)^{\varepsilon}\cdot  h^{\varepsilon}dxdt-2\int_{0}^{T}\int_{\mathbb{T}^3}\text{div}(h\otimes h)^{\varepsilon}\cdot h^{\varepsilon}dxdt\\
=&-2\int_{0}^{T}\int_{\mathbb{T}^3}\B[(h\times u)^{\varepsilon}-(h^{\varepsilon}\times u^{\varepsilon})\B] h^{\varepsilon}dxdt+2\int_0^T\int_{\mathbb{T}^3} \B[(h\otimes h)^\varepsilon-h^\varepsilon\otimes h^\varepsilon\B]\nabla h^\varepsilon dxdt\\
=&I+II,
\ea\end{equation}
 where we have used the divergence-free condition and the fact that $h^{\varepsilon}\cdot(u^{\varepsilon}\times h^{\varepsilon})=u^\varepsilon\cdot(h^\varepsilon\times h^\varepsilon)=0$.

 Now, to obtain the desired the magnetic helicity conservation, we need to take the limits in $\varepsilon\to 0$ in the above equation. Precisely, we have to show that the terms $I$, $II$ tend to zero as $\varepsilon\to 0$. Since the second term $II$ is exactly the same as \eqref{03.6} in proof of Theorem \ref{the1.2}, hence, we here only need to focus on the term $I$. Actually,
applying Lemma \ref{lem2.4} with  $u\in L^{3} (0,T;L^{3}(\mathbb{T}^{3}))$ and $ h\in L^{3}(0,T;\underline{B}^{\f13}_{3,VMO}(\mathbb{T}^3)),$ we see that $I$ tends to zero as $\varepsilon\rightarrow0.$
By arguing as what was done to prove \eqref{0.37}, we have
 $$
|II|\leq o(1).
$$
Thus, the proof of this theorem is completed.
 \end{proof}
 \section{Refined sufficient conditions of conserved quantities  in the  viscous E-MHD and H-MHD equations}
 In this section, making use of the energy  $L^{\infty}(0,T;L^{2}(\mathbb{T}^3))\cap L^{2}(0,T;H^{1}(\mathbb{T}^3)))$, we will reduce the proof of  Corollary \ref{coro1.5}-\ref{coro1.6} to the results in Theorem \ref{energythe1.1}-\ref{mhthe1.3}.
 \begin{proof}[Proof of Corollary \ref{coro1.6}]
	First, Notice that for any weak solutions of the system \eqref{visE-MHD}, there holds that $h   \in L^{\infty}(0,T;L^{2}(\mathbb{T}^3))\cap L^{2}(0,T;H^{1}(\mathbb{T}^3))$.

1.(Energy conservation) Multiplying the viscid E-MHD equations \eqref{visE-MHD} by $(h^\varepsilon)^\varepsilon$ and integrating the resualant over time-space, we have
\be\ba\label{5.1}
&\int_0^T\int_{\mathbb{T}^{d}}\f{1}{2}\f{d}{dt}|h^\varepsilon|^2dxdt+\int_0^T\int_{\mathbb{T}^d}|\nabla h^\varepsilon|^2dxdt+\int_0^T\int_{\mathbb{T}^{d}}\nabla\times\B[(\nabla\times h)\times h\B]^{\varepsilon}h^{\varepsilon}dxdt=0.
\ea\ee	
Recalling \eqref{c9} and $\vec{j}=\nabla \times h$, we have
\begin{equation}\label{5.2}\ba
	\int_0^T\int_{\mathbb{T}^{d}}\nabla\times\B[(\nabla\times h)\times h\B]^{\varepsilon}h^{\varepsilon}dxdt
	&=\int_0^T\int_{\mathbb{T}^{d}}\B[(\vec{j}\times h)^{\varepsilon}-(\vec{j}^{\varepsilon}\times h^{\varepsilon})\B]\nabla\times h^{\varepsilon}dxdt\\
	&=\int_0^T\int_{\mathbb{T}^{d}}\B[(\vec{j}\times h)^{\varepsilon}-(\vec{j}^{\varepsilon}\times h^{\varepsilon})\B]\vec{j}^{\varepsilon}dxdt
	\ea
\end{equation}
To control this term, first in light of H\"older inequality, we see that
\begin{equation}\label{5.3}
	\begin{aligned}
		|\int_0^T\int_{\mathbb{T}^{d}}\nabla\times\B[(\nabla\times h)\times h\B]^{\varepsilon}h^{\varepsilon}dxdt|\leq &	\|(\vec{j}\times h)^{\varepsilon}-(\vec{j}^{\varepsilon}\times h^{\varepsilon})\|_{L^{\f{r_1 r_2}{r_1+r_2}}(0,T;L^{\f{pq}{p+q}}(\mathbb{T}^d))}\|\vec{j}^\varepsilon\|_{L^{r_1}(0,T;L^p(\mathbb{T}^d))},
	\end{aligned}
\end{equation}
where $\f{2}{r_1}+\f{1}{r_2}=1$ and $\f{2}{p}+\f{1}{q}=1$.
Moreover, since $\vec{j}\in L^{r_1}(0,T;L^p(\mathbb{T}^d))$ and $h\in L^{r_2}(0,T;L^q(\mathbb{T}^d))$, then it follows from Lemma \ref{lem2.4} that
\begin{equation}\label{5.4}
	\|(\vec{j}\times h)^{\varepsilon}-(\vec{j}^{\varepsilon}\times h^{\varepsilon})\|_{L^{\f{r_1 r_2}{r_1+r_2}}(0,T;L^{\f{pq}{p+q}}(\mathbb{T}^d))}\to 0, \text{as}\ \varepsilon\to 0,
\end{equation}
which immediately implies
\begin{equation}
	\int_0^T\int_{\mathbb{T}^{d}}\nabla\times\B[(\nabla\times h)\times h\B]^{\varepsilon}h^{\varepsilon}dxdt\to 0, \text{as}\ \varepsilon\to 0.
\end{equation}
Hence, taking the limits in \eqref{5.1} as $\varepsilon\to 0$, we can conclude the desired energy conservation.

2.(Magnetic helicity conservation) First, recalling the indentity $\Delta h=\nabla (\Div h)-\nabla \times (\nabla \times h)$, we can get the magnetic vector potential equation for the viscous E-MHD equations \eqref{visE-MHD} as follows
\begin{equation}\label{5.6}
	H_t+(\nabla \times h)\times h+(\nabla \times h)+\nabla p=0,\ \ \Div H=0,
\end{equation}
where $H=\text{curl}^{-1}$. Then multiplying \eqref{visE-MHD} by $(H^\varepsilon)^\varepsilon$ and \eqref{5.6} by $(h^\varepsilon)^\varepsilon$, respectively, we have
\begin{equation}\label{5.7}
	\begin{aligned}
&	\f{d}{dt}\int_{\mathbb{T}^3}h^{\varepsilon}\cdot H^{\varepsilon}dx=
	\int_{\mathbb{T}^3}h_{t}^{\varepsilon}\cdot H^{\varepsilon}dx+
	\int_{\mathbb{T}^3}h^{\varepsilon}\cdot H_{t}^{\varepsilon}dx\\
	=&-\int_{\mathbb{T}^3}  \B[\nabla\times\B((\nabla \times h)\times h\B)^{\varepsilon}+\nabla \times (\nabla \times h)^\varepsilon\B] \cdot H^{\varepsilon}dx -
	\int_{\mathbb{T}^3}h^{\varepsilon}\cdot \B[\B( (\nabla \times h)\times h\B)^\varepsilon +\nabla \times h^\varepsilon  +\nabla p^{\varepsilon}\B]dx\\
	=&-2\int_{\mathbb{T}^3}\B[(\nabla \times h)\times h\B]^{\varepsilon}h^{\varepsilon}dx-2\int_{\mathbb{T}^3} (\nabla \times h)^\varepsilon \cdot h^\varepsilon dx\\
	=&-2\int_{\mathbb{T}^3}\B[\text{div}(h\otimes h)^{\varepsilon}-\Div(h^\varepsilon\otimes h^\varepsilon)\B]h^{\varepsilon}dx-2\int_{\mathbb{T}^3} (\nabla \times h)^\varepsilon \cdot h^\varepsilon dx\\
	=&2\int_{\mathbb{T}^3}\B[(h\otimes h)^{\varepsilon}-(h^\varepsilon\otimes h^\varepsilon)\B]\nabla h^{\varepsilon}dx-2\int_{\mathbb{T}^3} (\nabla \times h)^\varepsilon \cdot h^\varepsilon dx,	
		\end{aligned}
\end{equation}
where we have used \eqref{non2} and the divergence-free condition. Hence, integrating \eqref{5.7} over time-space, we can obtain
\begin{equation}\label{5.8}
	\begin{aligned}
	\int_0^T\int_{\mathbb{T}^3} \f{d}{dt} h^\varepsilon\cdot H^\varepsilon dxdt +2\int_0^T\int _{\mathbb{T}^3} (\nabla \times h)^\varepsilon\cdot h^\varepsilon dxdt=2\int_0^T\int_{\mathbb{T}^3} \B[(h\otimes h)^{\varepsilon}-(h^\varepsilon\otimes h^\varepsilon)\B]\nabla h^{\varepsilon}dxdt.	
	\end{aligned}
\end{equation}	
Now, to obtain the desired magnetic helicity conservation, it suffices to show that the term on the right-hand side of \eqref{5.8} tends to zero as $\varepsilon\to 0$. For the convenience of computation, we let this term by $R$.

(1) To control the term $R$, it follows from the H\"older inequality that
\begin{equation}\label{c19}
	\begin{aligned}
		|R|=& \B|2\int_{0}^{T}\int_{\mathbb{T}^3}\B[ (h\otimes h)^{\varepsilon}-h^{\varepsilon}\otimes h^{\varepsilon}\B]\nabla h^{\varepsilon}dxdt\B|	\\
		\leq&C\|(h\otimes h)^{\varepsilon}-h^{\varepsilon}\otimes h^{\varepsilon}\|_{L^{\f{2p}{p+1}}(0,T; L^{\f{2q}{q+1}}(\mathbb{T}^{3}))} \|\nabla h^{\varepsilon}\|_{L^{\f{2p}{p-1}}(0,T; L^{\f{2q}{q-1}}(\mathbb{T}^{3}))}
	\end{aligned}
\end{equation}
Since $\vec{j}\in L^{p}(0,T;L^{q}(\mathbb{T}^{3}))$ and $ h\in L^{\f{2p}{p-1}}(0,T;L^{\f{2q}{q-1}}(\mathbb{T}^{3})),$ making use of
Lemma \eqref{lem2.6}, we have
\begin{equation}\label{c20} \begin{aligned}
		&\|(h\otimes h)^{\varepsilon}-h^{\varepsilon}\otimes h^{\varepsilon} \|_{L^{\f{2p}{p+1}}(0,T;L^{\f{2q}{q+1}}(\mathbb{T}^{3}))}\leq C\varepsilon\|h\|_{L^p(W^{1,q}(0,T;\mathbb{T}^{3}))} \| h\|_{L^{\frac{2p}{p-1}}(L^{\frac{2q}{q-1}}(0,T;\mathbb{T}^{3}))},\\
		&\B\|\nabla h^\varepsilon\B\|_{L^{\f{2p}{p-1}}(0,T;L^{\f{2q}{q-1}}(\mathbb{T}^{3}))}\leq C\varepsilon^{-1}\|h\|_{L^{\f{2p}{p-1}}(0,T;L^{\f{2q}{q-1}}(\mathbb{T}^{3}))},\\ &\limsup_{\varepsilon\rightarrow0}\varepsilon\B\|\nabla h\B\|_{L^{\f{2p}{p-1}}(0,T;L^{\f{2q}{q-1}}(\mathbb{T}^{3}))}=0.
\end{aligned}\end{equation}
This  implies that
\begin{equation}\label{c21}
	\limsup_{\varepsilon\rightarrow0}
	\B|2\int_{0}^{T}\int_{\mathbb{T}^3}\B[ (h\otimes h)^{\varepsilon}-h^{\varepsilon}\otimes h^{\varepsilon}\B]\nabla h^{\varepsilon}dxdt\B|=0.\end{equation}
Thus, taking the limits in \eqref{03.6} as  $\varepsilon\rightarrow0$  and using the Lebesgue dominated convergence theorem and  \eqref{c21},  we complete the proof of this part.

(2) We derive from  $\vec{j}\in L^{p}(0,T;L^{q}(\mathbb{T}^{3}))$  that
	$\nabla h\in L^{p}(0,T;L^{q}(\mathbb{T}^{3}))$  via  the classical Calder\'on-Zygmund Theorem.
	A natural choice in the first part of Corollary \ref{coro1.6}  is $p=q=2$, hence, we know that $h\in L^{4}(0,T;L^{4}(\mathbb{T}^3))$ guarantee magnetic   helicity relation  \eqref{mhi}. Then, we have completed the second part of this Corollary.
	
	(3) In light of  the interpolation inequality, we have
	\begin{equation}
		\begin{aligned}
			\|h\|_{L^{4}(0,T;L^{4}(\mathbb{T}^3) )}  \leq& C\|h\|_{L^{\infty}\left(0,T;L^{2}(\mathbb{T}^{3})\right)}^{\frac{(q-4)}{2 q-4}}\|h\|_{L^{p}(0,T;L^q(\mathbb{T}^3))}^{\frac{q}{2q-4}}\leq C.
		\end{aligned}
	\end{equation}
	where $\frac{2}{p}+\frac{2}{q}=1$ and $q\geq 4$ was used.

	Likewise, we also have
	\begin{equation}
		\begin{aligned}
			\|h\|_{L^{4}(0,T;L^{4}(\mathbb{T}^3))}
			&\leq C\|h\|_{L^2(0,T;L^6(\mathbb{T}^3))}^{\frac{3(4-q)}{2(6-q)}}\|h\|_{L^p(0,T;L^q(\mathbb{T}^3))}^{\frac{q}{2(6-q)}}\\
			&\leq C\left(\|\nabla h\|_{L^2(0,T;L^2(\mathbb{T}^3))}+\|h\|_{L^\infty(0,T;L^2(\mathbb{T}^3))}\right)^{\frac{3(4-q)}{2(6-q)}}\|h\|_{L^p(0,T;L^q(\mathbb{T}^3))}^{\frac{q}{2(6-q)}}\leq C.
		\end{aligned}
	\end{equation}
	where $3<q<4$ and $\frac{1}{p}+\frac{3}{q}=1$ was used.
	
	Hence, based on the result in part (2), we conclude the  third part of this Corollary.
	
	(4) For  $\frac{1}{p}+\frac{6}{5 q}=1$ and $q\geq \frac{9}{5}$, we deduce from  the Gagliardo-Nirenberg inequality that
	\begin{equation}\label{c205.3}
		\begin{aligned}
			\|h\|_{L^{\f{2p}{p-1}}(0,T;L^{\f{2q}{q-1}}(\mathbb{T}^3))}\leq &C \|h\|_{L^{\infty}(0,T;L^{2}(\mathbb{T}^3))}^{\f{5q-9}{5q-6}}\|\nabla h\|_{L^{p}(0,T;L^{q}(\mathbb{T}^3))}^{\frac{3}{5q-6}}\\
			\leq &C \|h\|_{L^{\infty}(0,T;L^{2}(\mathbb{T}^3))}^{\f{5q-9}{5q-6}}\|\vec{j}\|_{L^{p}(0,T;L^{q}(\mathbb{T}^3))}^{\frac{3}{5q-6}}\leq C.
	\end{aligned}\end{equation}
	For $\frac{3}{2}<q<\frac{9}{5}$ and $\frac{1}{p}+\frac{3}{q}=2$, using the Gagliardo-Nirenberg inequality once again, we find
	\begin{equation}
		\begin{aligned}
			\|h\|_{L^{\frac{2p}{p-1}}(0,T;L^{\frac{2q}{q-1}}(\mathbb{T}^3))}&\leq C\|h\|_{L^2(0,T;L^6(\mathbb{T}^3))}^{\frac{9-5q}{6-3q}}\|\nabla h\|_{L^p(0,T;L^q(\mathbb{T}^3))}^{\frac{2q-3}{6-3q}}\\
			&\leq \left(\|\nabla h\|_{L^2(0,T;L^2(\mathbb{T}^3))}+\|h\|_{L^\infty(0,T;L^2(\mathbb{T}^3))}\right)^{\frac{9-5q}{6-3q}}\|\nabla h\|_{L^p(0,T;L^q(\mathbb{T}^3))}^{\frac{2q-3}{6-3q}}\\
			&\leq \left(\|\vec{j}\|_{L^2(0,T;L^2(\mathbb{T}^3))}+\|h\|_{L^\infty(0,T;L^2(\mathbb{T}^3))}\right)^{\frac{9-5q}{6-3q}}\|\vec{j}\|_{L^p(0,T;L^q(\mathbb{T}^3))}^{\frac{2q-3}{6-3q}}\leq C.
		\end{aligned}
	\end{equation}
	Then we conclude the desired result from the first part of this Corollary \ref{coro1.6}. Then the proof of this corollary is completed.
\end{proof}

 \begin{proof}[Proof of Corollary \ref{coro1.5}]
 (1) We derive from the $u\in L^{2}(0,T;H^{1}(\mathbb{T}^3))$ and Sobolev embedding theorem that $u\in L^{2}(0,T;H^{\f56}(\mathbb{T}^3))$. Since $H^{\f56}(\mathbb{T}^3)\approx B^{\f56}_{2,2}$, we conclude  by  $B^{\f56}_{2,2}\subseteq B^{\f13}_{3,2}$ that $u\in L^{2}(0,T; B^{\f13}_{3,2})$. This together with the Dominated-Convergence-Theorem and $ u\in L^{3}(0,T;B_{3,\infty}^{\f13})$
means that $ u\in L^{3}(0,T;B_{3,c(\mathbb{N})}^{\f13})$. This further enables us to get $u\in L^{3}(0,T;\underline{B}^{\f13}_{3,VMO}(\mathbb{T}^d))$. Theorem \ref{energythe1.1} helps us to complete the proof of this part.

(2) Exactly as in the above derivation, we deduce from $h\in L^{3}(0,T;B^{\f13}_{3,\infty}(\mathbb{T}^3))$ that $h\in L^{3}(0,T;\underline{B}^{\f13}_{3,VMO}(\mathbb{T}^d))$. The interpolation inequality and $u\in L^{\infty}(0,T;L^{2}(\mathbb{T}^3))\cap L^{2}(0,T;H^{1}(\mathbb{T}^3))$ leads to $u\in L^{3} (0,T;L^{3}(\mathbb{T}^{3}))$. With this in hand, we immediately finish the proof of this part by Theorem \ref{mhthe1.3}.
 \end{proof}

\section*{Acknowledgements}

 Wang was partially supported by  the National Natural
 Science Foundation of China under grant (No. 11971446, No. 12071113   and  No.  11601492)   and
sponsored by Natural Science Foundation of Henan Province (No. 232300421077).  Ye was partially supported by the National Natural Science Foundation of China  under grant (No.11701145) and sponsored by Natural
Science Foundation of Henan Province (No.
232300420111).

{\bf Author Declarations statement}

The authors have no conflicts to disclose.

{\bf Data Availability}

This publication is supported by multiple datasets, which are openly available at locations cited in the reference section.

\end{document}